\documentclass[a4paper,11pt]{article}
\usepackage{amsmath,amssymb,amsthm}
\usepackage{wasysym}
\usepackage{mathrsfs}
\usepackage{euscript}
\usepackage{color}
\usepackage[shortlabels]{enumitem}
\usepackage{bm}
\usepackage{charter}
\usepackage{comment}
\usepackage{srcltx}
\usepackage{xcolor}
\usepackage{graphicx}
\usepackage{amsmath}
\usepackage{amssymb}
\usepackage{euscript}
\usepackage{epic,eepic}
\usepackage{pstricks}
\usepackage{parskip}
\usepackage{tikz}
\usepackage{color}
\usepackage{multirow}
\usepackage{float}
\usepackage{multirow, array}
\usepackage{booktabs}
\usepackage{mathrsfs}
\usepackage{mathdots}
\usepackage{stackrel}
\usepackage{relsize}
\usepackage{ dsfont }
\usepackage{listings}
\usepackage{color}
\usepackage{graphicx}
\usepackage{setspace} 
\usepackage[utf8]{inputenc}
\usepackage{xcolor}
\usepackage[USenglish]{babel}
\usepackage{mathtools}
\usepackage{subfigure}
\usepackage{amsmath}
\usepackage{amssymb}
\usepackage{euscript}
\usepackage{epic,eepic}
\usepackage{pstricks}
\usepackage{parskip}
\usepackage{tikz}
\usepackage{multirow}
\usepackage{float}
\usepackage{array}
\usepackage{booktabs}
\usepackage{mathrsfs}
\usepackage{tabularx}
\usepackage{mathdots}
\usepackage{stackrel}
\usepackage{fancybox, calc}
\usepackage{verbatim}
\usepackage{fancyhdr}
\usepackage{tabularx}

\usepackage{enumitem} % Per personalizzare le liste
\usepackage{hyperref} % Per riferimenti ipertestuali
\usepackage{xstring} % Per gestire testi personalizzati nei riferimenti

\definecolor{mygreen}{rgb}{0,0.6,0}
\definecolor{mygray}{rgb}{0.5,0.5,0.5}
\definecolor{mymauve}{rgb}{0.58,0,0.82}
\definecolor{labelkey}{rgb}{0,0.08,0.45}
\definecolor{refkey}{rgb}{0,0.6,0.0}
\definecolor{Brown}{rgb}{0.45,0.0,0.05}
\definecolor{dgreen}{rgb}{0.00,0.49,0.00}
\definecolor{dblue}{rgb}{0,0.08,0.75}
\DeclareSymbolFont{largesymbolsA}{U}{jkpexa}{m}{n}
\SetSymbolFont{largesymbolsA}{bold}{U}{jkpexa}{bx}{n}
\DeclareMathSymbol{\varprod}{\mathop}{largesymbolsA}{16}

\definecolor{labelkey}{rgb}{0,0.08,0.45}
\definecolor{refkey}{rgb}{0,0.6,0.0}
\definecolor{Brown}{rgb}{0.45,0.0,0.05}
\definecolor{dgreen}{rgb}{0.00,0.49,0.00}
\definecolor{dblue}{rgb}{0,0.08,0.75}
\newtheorem{theorem}{Theorem}[section]
\newtheorem{corollary}[theorem]{Corollary}
\newtheorem{lemma}[theorem]{Lemma}

\newtheorem{proposition}[theorem]{Proposition}

\newtheorem{assumption}[theorem]{Assumption}
\theoremstyle{definition}
\newtheorem{definition}[theorem]{Definition}

\newtheorem{remark}[theorem]{Remark}
\newtheorem{example}[theorem]{Example}

\def\acknowledgement{\par\addvspace{17pt}\small\rmfamily
\trivlist\if!\ackname!\item[]\else
\item[\hskip\labelsep
{\bfseries\ackname}]\fi}
\numberwithin{equation}{section}

\providecommand{\norm}[1]{\lVert#1\rVert}

\newcommand{\R}{\ensuremath \mathbb{R}}

\newcommand{\N}{\ensuremath \mathbb{N}}

\DeclareMathOperator*{\dom}{\ensuremath{\text{\rm dom}}}
\DeclareMathOperator*{\argmin}{\text{\rm argmin}}

\newcommand{\customlabel}[1]{%
  A\arabic{enumi}'\protected@write\@auxout{}{%
    \string\newlabel{#1}{{A\arabic{enumi}'}{\thepage}}}}
\makeatother

\newcommand{\theu}{\ensuremath \theta_1}
\newcommand{\thed}{\ensuremath \theta_2}
\newcommand{\W}{\ensuremath W}
\newcommand{\w}{\ensuremath w}

\author{Cristian Vega}
\title{\sffamily Fast iterative regularization by reusing data
\footnote{Contact author: 
Cristian Vega Cere\~no, {\ttfamily cristian.vega.14@sansano.usm.cl},
}}
\author{Cristian Vega$^1$ 
\\[5mm]
\small MaLGa Center - DIMA: Universit`a degli Studi di Genova, Via Dodecaneso 35, Genoa, Italy. \\
}

\date{\ttfamily \today}
\newcommand{\scal}[2]{{\left\langle{{#1}, \ {#2}}\right\rangle}}

\newcommand{\LLL}{\ensuremath{\mathcal L}}

\newcommand{\Id}{\ensuremath{\operatorname{Id}}\,}
\newcommand{\RR}{\ensuremath{\mathbb{R}}}

\newcommand{\ran}{\ensuremath{\operatorname{ran}}}

\newcommand{\NN}{\ensuremath{\mathbb N}}

\newcommand{\vertiii}[1]{{\left\vert\kern-0.25ex\left\vert\kern-0.25ex\left\vert #1 
\right\vert\kern-0.25ex\right\vert\kern-0.25ex\right\vert}}
\newcommand{\normi}{\vert\kern-0.25ex\vert\kern-0.25ex\vert}

\renewcommand{\abstractname}{Abstract}
\ProvideTextCommand{\DJ}{OT1}{\raisebox{0.25ex}{-}\kern-0.4em D}
\renewcommand\theenumi{(\roman{enumi})}
\renewcommand\theenumii{(\alph{enumii})}

\numberwithin{equation}{section}
\usepackage{graphicx}

\def\cL{{\mathcal{L}}}

\def\cR{{\mathcal{R}}}

\newcommand{\be}{\begin{equation}}
\newcommand{\ee}{\end{equation}}
\newcommand{\br}{\begin{remark}}
\newcommand{\er}{\end{remark}}
\newcommand{\bq}{\begin{qu}}
\newcommand{\eq}{\end{qu}}
\newcommand{\bn}{\begin{enumerate}}
\newcommand{\en}{\end{enumerate}}
\newcommand{\bi}{\begin{itemize}}
\newcommand{\ei}{\end{itemize}}
\newcommand{\beas}{\begin{eqnarray*}}
\newcommand{\eeas}{\end{eqnarray*}}
\newcommand{\bea}{\begin{eqnarray}}
\newcommand{\eea}{\end{eqnarray}}

\newcommand{\bb}{\beta}

\newcommand{\Diag}{\operatorname{Diag}}
 \usepackage[backend=biber, style=ieee, citestyle=numeric-comp, sorting=nty]{biblatex}
\usepackage{csquotes}
\usepackage{chngcntr}
\begin{document}

\title{ {\sffamily{
 Optimization by Reparametrization:\\
Time-Warped Mirror Flows with Singular Geometry}}}

\author{
Cristian Vega\thanks{Departamento de Matemática, Universidad Técnica Federico Santa María,
Santiago, Chile ({\tt cristian.vegac@usm.cl}). This work was carried out while the author was at
MaLGa Center, DIMA - Dipartimento di Matematica, Università di Genova, Genoa, Italy.},
Cesare Molinari\thanks{MaLGa Center, DIMA - Dipartimento di Matematica,
Università di Genova, Via Dodecaneso 35, Genoa, Italy
({\tt molinari@dima.unige.it}, {\tt silvia.villa@unige.it})},
Lorenzo Rosasco\thanks{MaLGa Center, DIBRIS, Università di Genova, Genoa, Italy
({\tt lorenzo.rosasco@unige.it})}
\thanks{Istituto Italiano di Tecnologia, Genoa, Italy},
Silvia Villa\footnotemark[2]
}

\date{\ttfamily \today}

\maketitle
\begin{abstract}
Solving data-driven problems requires defining complex models and fitting them to data, neural networks being a motivating example. The fitting procedure can be seen as an optimization problem, which is often non-convex, and hence optimization guarantees are hard to derive. An opportunity is provided by viewing the model of interest as a redundant reparameterization---an overparameterization---of some simpler model for which optimization results are easier to achieve. In this paper, after formalizing the above idea, we revisit some recent results and derive new ones. In particular, we consider the gradient flow of some classes of linear overparameterizations and show that they correspond to suitable mirror flows on the original parameters. Our main contribution relates to the study of the latter, for which we establish well-posedness and convergence. Several specific instances are provided, including fully connected linear networks and networks with weight normalization. The results yield insight into the role of overparameterization in optimization and corresponding implicit regularization properties.
\end{abstract}

%\tableofcontents

%\maketitle
%\tableofcontents

\maketitle

\section{Introduction}
%%%%%%%%%%%%%%%%%

Data-driven approaches based on machine learning start from a parametric model
$f_\theta$, with $\theta\in\RR^k$. Given data
$\{(x_i,y_i)\}_{i=1}^n$ and a loss function $\ell$, the parameter is chosen by
minimizing the empirical error
$$
    \sum_{i=1}^n \ell(f_\theta(x_i),y_i),
$$
typically using a variant of stochastic gradient descent. The resulting
optimization problem is often nonconvex, and the parameter dynamics
$\theta(t)$ induced by gradient methods can be difficult to characterize.
The parameter trajectory may, however, be of limited interest, since what
matters is  usually the induced trajectory of the model $f_{\theta(t)}$.
Unfortunately, characterizing this trajectory is still difficult in general.
A starting point is the simpler setting in which $f_\theta$ is linear in the
input. More precisely, assume that there exists a surjective map
$q:\RR^k\to\RR^p$ such that

$$
    f_\theta(x)=\langle q(\theta),x\rangle.
$$

In this case, the model can equivalently be parameterized directly by
$\bb\in\RR^p$ as $f_\bb(x)=\langle\bb,x\rangle$; the map $q$ then
reparameterizes $\bb$ in terms of $\theta$. When $k>p$, or more generally when
$q$ is not injective, the reparameterization is redundant and hence defines an
overparameterization of the original model.

Parameter fitting can therefore be formulated in the original variable as
\begin{equation}
\label{probl_original}
\min_{\bb\in\RR^p}\cL(\bb),
\qquad
\cL(\bb):=\sum_{i=1}^n\ell(\langle\bb,x_i\rangle,y_i),
\end{equation}
or, using the reparameterization, as
\begin{equation}
\label{probl_ex}
\min_{\theta\in\RR^k}\cL(q(\theta)).
\end{equation}
Optimization by gradient methods can be idealized by considering gradient flow
on \eqref{probl_ex},
\begin{equation}
\label{eqqq}
\dot\theta(t)=-\nabla_\theta\cL(q(\theta(t))),
\end{equation}
for a suitable initialization. Studying the induced trajectory
$\bb(t)=q(\theta(t))$ in the original variable raises two key questions:

\begin{itemize}
\item Can the trajectory $\bb(t)$ be characterized? In particular, does it
converge to some $\bb_\infty\in\mathcal S$, where
\(\mathcal S:=\argmin_{\bb\in\RR^p}\cL(\bb)\)?
\item Is there a bias functional $\cR:\RR^p\to\RR$ such that
$\bb_\infty$ solves
\(\min_{\bb\in\mathcal S}\cR(\bb)\)?
If so, can $\cR$ be expressed explicitly in terms of the
reparameterization map $q$ and the initialization?
\end{itemize}

To place these questions in context, the relation between
\eqref{probl_ex} and \eqref{probl_original} can be read in two directions.

Starting from an overparameterized model, passing to the original
variable $\bb=q(\theta)$ allows one to study the optimization dynamics and
implicit bias induced by optimization in $\theta$
\cite{amid2020reparameterizing,li2022implicit,azulay2021implicit}. In
classical algorithm design, a model is combined with an explicit bias
reflecting prior knowledge about the problem \cite{hastie2001elements}. For
linear models, this bias can be encoded by a functional $\cR(\bb)$ on the
model parameters, such as a norm. By contrast, overparameterized models can
achieve excellent learning performance without an explicit regularizer. This
apparent puzzle can be understood by examining the interplay between the
reparameterization and the optimization procedure. Indeed, several examples
show that these two choices jointly induce an implicit bias on the original
parameters. This perspective has been considered in a number of recent works
\cite{gunasekar2018characterizing,gunasekar2018implicit,li2022implicit,
moroshko2020implicit,nacson2022implicit,neyshabur2014search,
vaskevicius2019implicit,woodworth2020kernel,arora2019implicit,
azulay2021implicit}.

Conversely, starting from a problem in the variable $\bb$, one can introduce a
map $q$ so that a standard algorithm applied to \eqref{probl_ex} induces more
favorable dynamics on the original problem. This perspective is useful beyond
machine learning: reparameterization can turn constrained problems into
unconstrained ones \cite{chou2022non}, improve the constants in convergence rates
\cite{labarriere2024optimization}, or replace nonsmooth problems by smooth ones
\cite{Poon2023}.

In both perspectives, the key idea is to relate gradient flow on the
reparameterized problem to mirror flow on the original one. By the chain
rule, omitting the dependence on time for readability,
$$
    \dot\bb
    =J_q(\theta)\dot\theta
    =-J_q(\theta)\nabla_\theta\cL(q(\theta))
    =-J_q(\theta)J_q(\theta)^\top\nabla\cL(\bb),
$$
where $J_q$ denotes the Jacobian of $q$. Suppose that there exists a function
$F:\RR^p\to\RR$ such that, for every $\theta\in\RR^k$,
\begin{equation}
\label{eq:jac}
J_q(\theta)J_q(\theta)^\top
=\left[\nabla^2F(q(\theta))\right]^{-1}.
\end{equation}
Then $\bb$ follows the mirror-flow dynamics
\begin{equation}
\label{eq:MF0}
\dot\bb=-\left[\nabla^2F(\bb)\right]^{-1}\nabla\cL(\bb).
\end{equation}
Thus, gradient flow in $\theta$ is equivalent to mirror flow in $\bb$, with a
geometry determined by the reparameterization through the mirror function
$F$. The corresponding implicit bias is therefore determined jointly by the
algorithm, the map $q$, and the initialization.

A general construction realizing continuous mirror descent as
gradient descent after reparameterization was given in
\cite[Theorem 1]{amid2020reparameterizing}. The converse perspective was
developed in \cite{li2022implicit}, which characterizes the equivalence for
commuting parameterizations. This framework applies, in particular, to
diagonal linear networks, for which the induced bias is related to sparsity in
the original linear model
\cite{chou2023more,woodworth2020kernel,zhao2019implicit}. Component-wise
reparameterizations are studied systematically in \cite{chou2023induce}; in
the appendix, we also give explicit computations of $F$ for two examples from
\cite{chou2023more,li2022implicit} (see Examples \ref{Ex: uL} and
\ref{Ex: UV}).

The mirror-flow representation is nevertheless restrictive. To obtain
\eqref{eq:MF0}, one must first express
$J_q(\theta)J_q(\theta)^\top$ solely in terms of the global variable
$\bb=q(\theta)$ and then find a function $F$ whose inverse Hessian equals the
resulting matrix field. For some relevant reparameterizations, such an $F$ does not
exist. The following example, introduced in \cite{azulay2021implicit}, is used throughout the paper, studied in detail
in Section \ref{SS: AFCLN} and illustrates this obstruction.

\begin{example}[Multi-neuron fully connected linear network of depth 2]
\label{Ex: MV}
Consider $\bb=q(\theta)$ with
$\theta=(\W,\w)\in\RR^{p\times m}\times\RR^m$ and
$q(\theta)=\W\w$. Gradient flow on the reparameterized problem is
\begin{align}
\label{eq:multineuron}
\dot\W(t)
&=-\nabla_\W\cL(\W(t)\w(t))
=-\nabla\cL(\bb(t))\w(t)^\top,\\
\dot\w(t)
&=-\nabla_\w\cL(\W(t)\w(t))
=-\W(t)^\top\nabla\cL(\bb(t)).\nonumber
\end{align}
Assume, for simplicity, that the initialization satisfies
$\W(0)^\top\W(0)=\w(0)\w(0)^\top.$

By the Poincar\'e lemma (see Lemma \ref{poincare}), there is no function $F$
satisfying \eqref{eq:jac}. Nevertheless, as observed in
\cite{azulay2021implicit}, the induced dynamics in the original variable can
be written as
\begin{equation}
\label{eq: Uvanot0 Flow H1}
\left[\nabla^2F(\bb(t))\right]\dot\bb(t)
=-G(\bb(t))\nabla\cL(\bb(t)),
\end{equation}
where
\begin{equation}
\label{eq:F}
F(\bb)=\frac{2}{3}\|\bb\|^{3/2},
\qquad
G(\bb)=\|\bb\|^{1/2}.
\end{equation}
See Section \ref{SS: AFCLN} for a proof.
\end{example}

Equation \eqref{eq: Uvanot0 Flow H1} differs from standard mirror flow in two
important respects. First, it contains the state-dependent time-warping factor
$G$. Second, the mirror function in \eqref{eq:F} is not twice differentiable at
the origin. The global well-posedness requires showing that
the trajectory remains away from this singular point. So, since $G$ is scalar and positive away from the origin, it changes only the
speed along the mirror-flow trajectory. This prevent to simply apply classical results on mirror flow \cite{alvarez2004hessian}.
 Time changes of Hessian flows have also been studied as
conformal mirror descent \cite{kainth2022conformal}, while classical
well-posedness and convergence results for Hessian Riemannian flows are given
in \cite{alvarez2004hessian}. These results do not directly cover the interior
singularity and reparameterization-induced time factors considered here.
Motivated by this example, we study the generalized mirror flow (GMF)
\eqref{eq: Uvanot0 Flow H1}. Our main contributions are as follows.

\begin{itemize}
\item We establish a general framework for generalized mirror
flow under assumptions that allow the Hessian of the mirror function to be
undefined at one point. We prove global existence and uniqueness, show that
the loss converges to its global minimum with a last-iterate rate
$\mathcal O(1/t)$, and prove convergence of the entire trajectory.

\item For a general convex loss, we give a variational
characterization of the selected minimizer. When the loss is the composition
of a convex function and a linear operator and its minimizer set is affine,
the limit is the Bregman projection of the initialization onto that set,
yielding an explicit characterization of the implicit bias.

\item We apply the framework to reparameterizations for which
the classical mirror-flow equivalence fails in the original time scale. For
depth-two fully connected linear networks, we recover the
time-warped representation and implicit-bias functional identified in
\cite{azulay2021implicit} and place them under global well-posedness and
convergence guarantees. For radial weight normalization, we characterize the
relation between the norm reparameterization and the induced mirror geometry,
recover polynomial and exponential schemes as special cases, and analyze a
combined weight-normalized fully connected construction
\cite{morwani2022inductive,salimans2016weight,wu2020implicit,
chou2023robust}.
\end{itemize}

The paper is organized as follows.  In Section \ref{S: Main result}, we study
the well-posedness and convergence properties of the generalized mirror flow
in \eqref{eq: Uvanot0 Flow H1}. In Section \ref{SS: AFCLN}, we apply these
results to depth-two fully connected linear networks. Section \ref{Sec: WN}
considers radial mirror functions and time-warping factors, leading to a general
analysis of weight-normalization schemes and their implicit bias. Section
\ref{sec:conclusion} concludes the paper. The appendices collect the notation,
preliminary results, and proofs, and extend the analysis to
Legendre mirror functions defined on proper subsets of $\RR^p$.

Before developing the analysis, we discuss the closest related
work and clarify the scope of our contributions.

\subsection{Related work}\label{sec:relwork}
%%%%%%%%%%%%%%%%%%%%%%%%%%%%%
We review the works closest to the two aspects of our
contribution: the representation of reparameterized gradient flow as 
intrinsic dynamics in the original variable, and the implicit bias induced by
specific linear overparameterizations.

\noindent\textbf{Mirror-flow representations and intrinsic dynamics.}
Amid and Warmuth \cite{amid2020reparameterizing} give a general construction
that realizes continuous mirror descent as gradient descent after a change of
parameters. Li et al.\ \cite{li2022implicit} characterize the converse
correspondence for commuting parameterizations and establish convergence and
implicit-bias results in that setting. Further developments for linear and
component-wise parameterizations include
\cite{chou2023more,chou2023induce,chou2022non,jacobs2025mask,
pesme2021implicit}.
Existence and convergence of Hessian Riemannian gradient flows are studied in
the classical work \cite{alvarez2004hessian}. State-dependent time changes of
such flows also appear in conformal mirror descent
\cite{kainth2022conformal}, where the conformal factor is induced by
logarithmic-divergence geometry. More broadly, the question of when gradient
flow in a high-dimensional parameter induces an intrinsic lower-dimensional
dynamics is studied in \cite{marcotte2025intrinsic}, building on the
conservation-law framework of \cite{marcotte2023abide}. This intrinsic
Riemannian viewpoint is geometrically more general than the conformally
Hessian dynamics considered here.

Our analysis concerns reparameterizations for which the induced dynamics are
not a standard mirror flow in the original time scale. For the matrix-times-vector model studied in Section \ref{SS: AFCLN}, the inverse
metric fails the Hessian integrability condition but becomes a Hessian after
multiplication by a scalar state-dependent factor, as observed in
\cite{azulay2021implicit}. We provide  global well-posedness and convergence
analysis for the resulting time-warped flow when the mirror function may fail
to be twice differentiable at an interior point.

\noindent\textbf{Diagonal linear networks.}
Diagonal linear networks provide the simplest setting in which
reparameterized gradient flow induces a non-Euclidean geometry and a
sparsity-promoting bias. For least squares, the dependence of this bias on
depth and initialization is studied in
\cite{zhao2019implicit,woodworth2020kernel,chou2023more,
labarriere2024optimization}; in the small-initialization regime, the selected
solution approaches a minimum-$\ell_1$ solution, at the cost of slower
optimization. Component-wise reparameterizations and the inverse problem of
designing a reparameterization for a prescribed regularizer are treated
systematically in \cite{chou2023induce}. Stochastic gradient methods can
produce a stronger sparsity bias than their deterministic counterparts
\cite{ali2020implicit,pesme2021implicit}. Related geometric analyses for deep
convolutional networks are given in \cite{achour2025riemannian}.

\noindent\textbf{Fully connected linear networks.}
The matrix-times-vector parameterization underlying a depth-two fully
connected linear network has been studied from several complementary
perspectives. Azulay et al.\ \cite{azulay2021implicit} introduce the
time-warped mirror-flow representation used here and derive the corresponding
implicit-bias functional conditionally on convergence to an interpolating
solution. Min et al.\ \cite{min2021explicit} study the role of initialization
imbalance under square loss, proving exponential loss convergence and
minimum-norm bias under suitable initialization conditions. Varre et al.\
\cite{varre2023spectral} establish convergence and a variational
characterization for a particular initialization shape and expose a mirror
flow governing the singular values. Kunin et al.\ \cite{kunin2024rich}
extend time-warped implicit-bias calculations to unbalanced, wide, and deep
linear networks, again under trajectory-avoidance and interpolation
conditions. For the balanced multi-neuron model considered here, our general
theory yields global well-posedness and convergence for convex losses
satisfying our assumptions and recovers the bias of
\cite{azulay2021implicit} as an exact Bregman projection in the
affine-composite case.

\noindent\textbf{Weight normalization.}
Weight normalization separates the magnitude and direction of a parameter
\cite{salimans2016weight}. For overparameterized least squares, Wu et al.\
\cite{wu2020implicit} analyze standard weight normalization in continuous and
discrete time and establish convergence and minimum-norm or
near-minimum-norm guarantees in different regimes. Standard and exponential
weight normalization for homogeneous classifiers are studied in
\cite{morwani2022inductive}, while \cite{chou2023robust} combines weight
normalization with a diagonal linear model to obtain a robust sparsity bias.
Our results complement these works: we characterize a general class of radial
norm reparameterizations in continuous time, identify the compatibility
relation between the scale map and the induced mirror geometry, and obtain an
exact Bregman-projection bias for affine-composite losses. Polynomial and
exponential weight normalization arise as special cases, and the framework
also covers their combination with the fully connected factorization.

\noindent\textbf{Classification and homogeneous models.}
Implicit bias has also been studied extensively for losses whose infimum is
not attained at a finite parameter. For linearly separable classification,
gradient descent on logistic or exponential losses converges in direction to
the Euclidean maximum-margin classifier
\cite{soudry2018implicit,nacson2019convergence,ji2018gradient,
ji2019implicit}; stochastic-gradient extensions are given in
\cite{nacson2019stochastic}. Corresponding results for diagonal linear
networks, linear convolutional networks, and more general homogeneous models
appear in \cite{moroshko2020implicit,gunasekar2018implicit,lyu2019gradient}.
The infinitely wide two-layer ReLU setting is studied in
\cite{chizat2020implicit}, and extensions involving weight normalization and
mirror flow are developed in
\cite{morwani2022inductive,pesme2024implicitbiasmirrorflow,
jacobs2026implicitbiasmirrorflow}. These works concern directional or
maximum-margin limits, whereas our results address convergence to finite
minimizers of a convex loss.

\noindent\textbf{Matrix factorizations and deep linear models.}
Related implicit-bias phenomena arise when a matrix is represented as a
product of factors. For commuting positive-semidefinite sensing matrices and
small initialization, gradient flow can select minimum-nuclear-norm solutions
\cite{gunasekar2017implicit,li2018algorithmic}; deeper factorizations and
their spectral bias are studied in \cite{arora2019implicit}. Results under
restricted-isometry assumptions appear in
\cite{wu2020continuous,vaskevicius2019implicit,you2020robust}, while
mirror-descent and tensor formulations are considered in
\cite{wu2021implicit,wang2021large,yun2020unifying}. Further optimization and
geometric analyses of deep linear networks include
\cite{arora2018optimization,maly2021robust,gidel2019implicit,
gissin2019implicit,menon2023geometry}.

\section{Global existence and asymptotic analysis} 
\label{S: Main result}
In the same setting of the Introduction, given a differentiable convex loss function $\mathcal{L}: \ \RR^p \to \R$, $\bb\mapsto \mathcal{L}(\bb)$, we consider the trajectory $\bb(t)$ generated by the gradient flow dynamics on a reparameterization $\bb=q(\theta)$. We analyze the case in which the gradient flow on the reparameterization $\theta$ may not be expressed as a mirror flow on the original variable $\bb$, but it can still be written as a dynamical system involving the inverse of the Hessian of $F$ and a positive function $G$, as in \eqref{eq: Uvanot0 Flow H1}:
\begin{equation}
\tag{GMF}
\label{eq: x(t)flow}
\begin{cases}
\left [\nabla^{2}F(\bb(t))\right]\dot{\bb}(t)=-G(\bb(t))
\nabla\mathcal{L}(\bb(t)), \quad\quad  t>0\\
\bb(0)=\bb_0\in \RR^p.
\end{cases}
\end{equation}
In the case when $G(\bb)\equiv 1$, the dynamical system \eqref{eq: x(t)flow} is indeed a mirror flow.  
\begin{remark}[Relation to conformal mirror descent]
On any interval where $G(\bb(t))>0$, the time change
\(
    s(t)=\int_0^t G(\bb(u))\,du
\)
transforms \eqref{eq: x(t)flow} into the mirror flow
\[
    \frac{d\widetilde{\bb}}{ds}
    =-\left[\nabla^2F(\widetilde{\bb})\right]^{-1}
      \nabla\mathcal{L}(\widetilde{\bb}),
    \qquad \widetilde{\bb}(s(t))=\bb(t).
\]
This connection is established for conformal mirror descent
in \cite[Theorem~2]{kainth2022conformal}
under suitable assumptions on the relation between $F$ and $G$
that we do not require. In our case, this point of view does not seem advantageous with respect to a direct analysis of \eqref{eq: x(t)flow}.
\end{remark}

We present a comprehensive framework to study the existence and uniqueness of the trajectory of \eqref{eq: x(t)flow} and its minimization properties, such as the decrease of the loss' values along it and the corresponding convergence rates (see Theorems \ref{T: Existence} and \ref{T: Main result}). Moreover, we show the convergence of the trajectory to $\mathcal{S}=\argmin \mathcal{L}$ and we explicitly compute its implicit bias towards a specific solution in $\mathcal{S}$ (see Theorem \ref{T: Main result}). 
With respect to mirror flow, our analysis shows that for the system \eqref{eq: x(t)flow} the function $G$ does not affect the decrease of the loss in time, and the convergence of the trajectory. 

Our analysis generalizes previous works in different directions: first, 
we prove the well-posedness and the convergence over time of the trajectory $\bb(t)$ of the dynamical system \eqref{eq: x(t)flow}, while in \cite{azulay2021implicit,arora2018optimization,li2022implicit} these properties are assumed a priori. Additionally, we derive the implicit bias induced by reparameterization for a general convex loss function $\mathcal{L}$, while in previous work only the least squares loss function has been considered. For the least squares loss, we recover that the implicit bias $\mathcal{R}$ is the Bregman divergence of the entropy $F$ as in {\color{green}\cite{amid2020reparameterizing,li2022implicit}} (see Corollary~\ref{Remark: Saving energy}). \\ \\

\subsection{Well-posedness}
In this section we show that the system \eqref{eq: x(t)flow} is well-posed by proving the existence of a unique solution that is defined in the entire time-interval $[0, +\infty)$. To do so, we first list the assumptions that we require on the loss function $\mathcal{L}$, the initialization $\bb(0)=\bb_0$, the mirror map $F$ and the time-warping function $G$. These hypotheses are also useful to verify, in Theorem \ref{T: Main result}, the convergence of the global variable $\bb(t)$ to a specific minimizer of the loss function. Define
\[
\mathcal{S}:=\argmin_{\bb \in \R^p}\mathcal{L}(\bb) \quad \quad \text{and} \quad \quad \mathcal{L}^*:=\min_{\bb \in \R^p}\mathcal{L}(\bb).
\]
\\
\begin{assumption}\label{Assumption}\ \\
\ \\
Let $b\in\RR^p$. 
\begin{enumerate}[label=A\arabic*]
\item\label{A: D1} The function $\mathcal{L}\colon \ \RR^{p}\to\RR$ is convex, differentiable and $\nabla \LLL$ is locally Lipschitz continuous. Moreover, $\mathcal{S} \neq \emptyset$ and $b \notin \mathcal{S}$.
\item\label{A: D3} The initialization satisfies $\mathcal{L}(\bb_0) < \mathcal{L}(b)$. 
\item\label{A: D2} The function $F\colon \RR^p \to \RR$ is strictly-convex and twice-differentiable on $\RR^{p}\setminus\left\{b\right\}$. For every $\bb \in \RR^{p}\setminus\left\{b\right\}$, $\nabla^{2}F(\bb)$ is invertible and the map $\bb\in\RR^{p}\setminus\left\{b\right\} \mapsto (\nabla^{2}F(\bb))^{-1}$ is locally Lipschitz continuous. Moreover, the function $G: \ \RR^{p}\setminus\left\{b\right\} \to \RR_{+}$ is locally Lipschitz continuous. For every radius $R>0$ there exists a constant $C(R)>0$ such that, for all $\bb\not\in\mathbb{B}(b,R)$, $G(\bb)\geq C(R)$.
\item\label{A: D4} There exists $\bar{z}\in \mathcal{S}$ such that the function $D_{F}(\bar{z},\cdot)$ is coercive, where $D_{F}$ is the Bregman divergence associated with $F$ (see \cite{bregman1966relaxation}), defined as
\begin{align}
    D_{F}: \ \RR^{p}\times\RR^{p}\setminus\left\{b\right\} \to \RR, \quad D_F(y,\bb)= F(y)-F(\bb)-\scal{\nabla F(\bb)}{y-\bb}.
  \label{Def:BDA}
\end{align}
\end{enumerate}
\end{assumption}\vspace{4mm}

\begin{remark} Assumption~\ref{A: D3} implies that $\bb_0\neq b$ and $b\notin \mathcal{S}$, therefore $\nabla\mathcal{L}(b)\neq 0$.
\end{remark}

\begin{remark} Assumption~\ref{Assumption} is key to establish the well-posedness and the convergence analysis of the flow given by \eqref{eq: x(t)flow}. 
 Assumptions~\ref{A: D2} and ~\ref{A: D1} are standard in the classical theory of local existence and uniqueness of solutions of ordinary differential equations \cite[Subsection 1.2]{barreira2012ordinary}. 
We require convexity of the loss function $\mathcal{L}$, which is important to derive global convergence results. A similar equation has been considered under different assumptions  in \cite[Section 3.2]{alvarez2004hessian}. See also Section~\ref{sec:legendre}.
As we will see Assumptions~\ref{A: D3} ensures a suitable initialization for the flow, which avoids convergence of the solution towards $b$. 
% Condition~\ref{A: D3} is the generalization of the one presented in \cite[Lemma 2.3]{wu2020implicit} for least squares problems.  
Assumption~\ref{A: D4} is crucial to establish boundedness of the trajectory, which is an intermediate step for proving convergence. 
\end{remark} The following theorem establishes the existence and uniqueness of the global solution of the dynamical system \eqref{eq: x(t)flow}.
\begin{theorem}\label{T: Existence}
Under Assumption \ref{Assumption}, the dynamical system \eqref{eq: x(t)flow} has a unique solution $\bb$ defined on $[0,+\infty)$.
Moreover, for every ${z}\in \mathcal{S}$, both the Bregman divergence $D_{F}(z,\bb(t))$ and $\mathcal{L}(\bb(t))$ are non-increasing functions of $t\geq 0$.
\end{theorem}
 The proof is given in the Appendix; see Subsection \ref{Append: Well-posedness}. Theorem~\ref{T: Existence} provides sufficient conditions to ensure the well-posedness of the system on $[0,+\infty)$. This is fundamental to guarantee that the implicit bias characterization is meaningful. % Indeed, it is well-known that the trajectory of a differential equation exists globally if and only if it is locally bounded. 
 Indeed, if the solution is defined on a bounded interval and not on $[0,+\infty)$, there is a finite time for which the norm of the trajectory diverges. In this case, the limit of the trajectory for that finite time cannot have any interesting properties since it clearly cannot converge to some specific minimizer of $\mathcal{L}$. Our result generalizes the ones in \cite{chou2023induce} and \cite{chou2023robust}, which hold for specific examples. 

 % Indeed,  Theorem \ref{T: Existence} and Theorem \ref{T: Main result} would still hold replacing Assumption~\ref{A: D4}  by boundedness of the trajectory.  Observe that the coercivity hypothesis could be replaced for instance by the boundedness of the level set of the loss function. 
 
 % the authors show the existence of solution only in the specific case $G \equiv 1$ and $q$ separable. Similarly,  demonstrates well-posedness for weight normalization by modifying the gradient flow of the reparameterization to achieve a preconditioned gradient flow over the global variable. But none of them study time-warping mirror flow. 
 \subsection{Minimization properties and implicit bias}
 The next theorem illustrates the minimization properties of the solution of the dynamical system \eqref{eq: x(t)flow} for $t\to+\infty$. It states the convergence rates for the loss function and the convergence of the trajectory towards a specific minimizer of the loss, characterized by an implicit bias. 
 
\begin{theorem}\label{T: Main result}
Assume that Assumption \ref{Assumption} holds and that $\bb$ is the unique solution of the dynamical system \eqref{eq: x(t)flow}. Then $\mathcal{L}(\bb(t))$ is non-increasing and $\mathcal{L}(\bb(t))\to \mathcal{L}^*$ for $t\to+\infty$. In addition, $\left[\mathcal{L}(\bb(t))-\mathcal{L}^{*}\right]\in L^{1}([0,+\infty))$, and, for every ${z}\in \mathcal{S}$ and $t > 0$,
\begin{align*}
  \LLL(\bb(t))-\LLL^{*} \quad \leq \quad \frac{D_{F}(z,\bb_0)}{C(r) t },
\end{align*}
where $C$ is the function introduced in Assumption \ref{A: D2} and $r:=\frac{\mathcal{L}(b)-\mathcal{L}(\bb_0)}{2 \|\nabla \mathcal{L}(b)\|}$.
Moreover, $\lim_{t\rightarrow +\infty}\bb(t)=\bb_{\infty}$ where
\begin{align}\label{minbreg1}
\left\{\bb_{\infty}\right\} \ = \ \argmin\limits_{z\in\mathcal{S}} \left\{ D_{F}(z,\bb_0)+\int_{0}^{+\infty}G(\bb(t))\nabla\mathcal{L}(\bb(t)), z \big\rangle dt\right \}.
\end{align}
\end{theorem}
\vspace{2mm}
\noindent The proof is given in the Appendix; see Subsection \ref{Append: Main result}. % In the case $U=\RR^p\setminus\{b\}$, the function $F$ may be not twice differentiable at $b$. Therefore the proof techniques are different from the ones in \cite{alvarez2004hessian}, where the authors study the vanilla mirror flow trajectory on a convex set.
%  . Therefore the results in \cite{alvarez2004hessian} cannot be generalized in a straightforward manner to deal with the time-warping mirror flow in our setting. We tackle the nonconvexity of $U$, relying on Assumption \ref{A: D3i} and Lemma \ref{L: Bounded away x(t)}, which proves that $\|\bb(t)\|$ remains bounded away from $0$, thereby avoiding the point where the Hessian of $F$ is not defined which may be a stationary point, but not a minimizer of the loss function. 
The next corollary characterizes the implicit bias for time-warping mirror flow when the loss function is the composition of a convex function with a linear operator. This implicit bias corresponds to the Bregman distance between the set of minimizers of the loss and the initialization, extending the results presented in \cite{azulay2021implicit,wu2020implicit}.
\begin{corollary}
\label{Remark: Saving energy} Suppose that Assumption \ref{Assumption} holds and let  $\bb$ be the unique solution of the dynamical system \eqref{eq: x(t)flow}. Let $X\in \RR^{d\times p}$ and let $\ell\colon\RR^d\to \RR$ be differentiable. Assume that $\LLL=\ell \circ X$ and $\mathcal{S} = \left\{\bb\in\RR^{p}\mid X\bb=y\right\}$ for some $y\in\RR^{d}$. Let $P: \ \RR^p \to \RR^p$ be the projection operator onto the kernel of $X$. Then, for every $t>0$,
\begin{align}
\frac{d}{dt}\left(P(\nabla F\left(\bb(t)\right)\right)&=0.
\nonumber
\end{align}
Moreover, $\lim\limits_{t\rightarrow+\infty}\bb(t)=\bb_{\infty}$ with
\begin{align*} \{ \bb_{\infty} \} \ = \ \argmin\limits_{z\in\mathcal{S}} \ D_F(z,\bb_0). 
\end{align*}
\end{corollary}
\begin{proof}
Using linearity of $P$, the dynamical system and the fact that $\ran(X^{\top})=\ker(X)^{\perp}$, we get
\begin{align*}
\frac{d}{dt}\left(P(\nabla F\left(\bb(t)\right)\right)& = P\left(\nabla^2 F\left(\bb(t)\right)\right)\dot{\bb}(t)\\
&=P\left(-G(\bb(t))\nabla \mathcal{L}\left(\bb(t)\right)\right)\dot{\bb}(t) \\
&= -G(\bb(t))P\left(X^{\top}\nabla \mathcal{\ell}\left(X\bb(t)\right)\right) \dot{\bb}(t)= 0.
\end{align*}
It follows from Theorem \ref{T: Main result} that $\lim\limits_{t\rightarrow+\infty} \bb(t)=\bb_{\infty}$ with 
\begin{align*}
\left\{ \bb_{\infty}\right\} = \argmin\limits_{z\in\mathcal{S}} \left\{D_F(z,\bb_0)+\int_{0}^{+\infty}G(\bb(t))\scal{\nabla\mathcal{L}(\bb(t))}{z-\bb(t)}dt \right\}.
\end{align*}
To conclude the assertion, it is sufficient to prove that the integral is constant with respect to $z$. By the chain rule, for every $z\in\mathcal{S}$,
\begin{align*}
\int_{0}^{+\infty}G(\bb(t))\scal{\nabla\mathcal{L}(\bb(t))}{z-\bb(t)}dt=&\int_{0}^{+\infty}G(\bb(t))\scal{X^{\top}\nabla \ell(X\bb(t))}{z-\bb(t)}dt\nonumber\\=&\int_{0}^{+\infty}G(\bb(t))\scal{\nabla \ell(X\bb(t))}{y-X\bb(t)}dt
\end{align*}
which is independent of $z\in\mathcal{S}$. Consequently, from \eqref{minbreg1}, we obtain that the limit of the trajectory satisfies
\begin{align*}
   \left\{ \bb_{\infty} \right\} = \argmin\limits_{z\in\mathcal{S}} \quad D_F(z,\bb_0).
\end{align*}\hfill
\end{proof}
\begin{remark}
Consider the case of $\LLL=\ell \circ X$ with $X$ linear operator. A sufficient condition for the set of minimizers of $\LLL$ to be of the form $\mathcal{S} = \left\{\bb\in\RR^{p}\mid X\bb=y\right\}$ for some $y\in\RR^{d}$ is to assume that $\ell$ is strictly-convex. Indeed, denote by $y$ the unique minimizer of the problem $$\min_{p\in \ran(X)} \ell(p).$$
If $X\bar{\beta}=y$ then $\bar{\beta} \in \argmin \LLL$. Now we want to show that if $\bar{\beta} \in \argmin \LLL$ then $X\bar{\beta}=y$. We have $$\ell(X\bar{\beta})=\LLL(\bar{\beta})=\min_{\beta\in\R^p} \LLL(\beta)=\min_{\beta \in \R^p} \ell(X\beta)=\min_{p\in \ran(X)} \ell(p)=\ell(y).$$
Since $y$ is the unique minimizer of $\ell$ on $\ran(X)$, $X\bar{\beta}\in\ran(X)$ and $\ell(X\bar{\beta})=\ell(y)$, then  $X\bar{\beta} = y$.  
\end{remark}
\ \\
\ \\
% The point $\bb_{\infty}$ is  the (unique) Bregman projection, with respect to the entropy $F$, of $\bb_0$ on $\mathcal{S}$. In particular, $\bb_{\infty}$ belongs to the domain of the entropy $F$, which therefore can also implicitly enforce some constraints. For instance, 
% the parameterization $\bb=\theta\odot\theta$ implicitly enforces the positivity of the solution: in this case the domain of $F$ is the positive orthant, see the case $L=2$ in equation~\eqref{eq:had}. 
\vspace{2mm}
\section{Application to fully connected linear networks}\label{SS: AFCLN}
We consider the setting of Example \ref{Ex: MV}, namely multi-neuron fully connected linear networks of depth 2, see \cite{arora2019implicit,azulay2021implicit}.  In this case, $\bb=q(\theta)$ with $\theta=(\W,w)\in\RR^{p\times m} \times\RR^{m}$ and $q(\theta)=\W \w$. The gradient flow on the reparameterization is then given by
\begin{equation}\label{eq: UVector-Flow v-WNN}
\begin{cases}
     \Dot{\W}(t)&=-\nabla_{\W}\mathcal{L}(\W(t)\w(t))=-\nabla \mathcal{L}(\bb(t)) \w^{\top}(t)\\ 
     \Dot{\w}(t)&=-\nabla_{\w}\mathcal{L}(\W(t)\w(t))=- \W^{\top}(t)\nabla\mathcal{L}(\bb(t)) \hspace{6mm}\text{ with }\W(0)=W_0, \w(0)=w_0.
  \end{cases}    
\end{equation}

The following theorem  shows that the  dynamics in the variable $\beta$ is a generalized mirror flow.

\begin{theorem}
    
Let  Assumption~\ref{A: D1} and \ref{A: D3} hold with $b=0$ and  $(\W(t),\w(t))$ be the solution of \eqref{eq: UVector-Flow v-WNN} with
\begin{align}
\W^{\top}_0\W_0=\w_0\w^{\top}_0.
\label{Balance initialization}
\end{align}
Then $\bb(t)=\W(t)\w(t)$ satisfies \eqref{eq: x(t)flow} with $F(\bb)=(2/3)\|\bb\|^{3/2}$ and $G(\bb)=\|\bb\|^{1/2}$. 
\end{theorem}
\begin{proof}
If we multiply the expression of $\Dot{\W}(t)$ on the left by $\W^{\top}(t)$ and the one of $\Dot{\w}(t)$ on the right by $\w^{\top}(t)$, we get that
\begin{align}
      \W^{\top}(t)\Dot{\W}(t)=- \W^{\top}(t)\nabla\mathcal{L}(\bb(t))\w^{\top}(t)=\Dot{\w}(t)\w^{\top}(t).\nonumber
\end{align}
Adding to the previous equality the one obtained by transposing it, we have that,
\begin{align}
     \frac{d}{dt}\left(\W^{\top}(t)\W(t)\right)=\W^{\top}(t)\Dot{\W}(t)+\dot{\W}^{\top}(t)\W(t)=\Dot{\w}(t)\w^{\top}(t)+\w(t)\Dot{\w}^{\top}(t)=\frac{d}{dt}\left(\w(t)\w^{\top}(t)\right).\nonumber
\end{align}  
Then, from the latter equation and the initialization $\W^{\top}(0)\W(0)=\w(0)\w^{\top}(0)$, we deduce that, for every $t\geq 0$, 
\begin{align}
 \W^{\top}(t)\W(t)=\w(t)\w^{\top}(t).\label{eq: Matrix-eqq}\end{align}
Multiplying by $\w(t)$ on the right, by $\w^{\top}(t)$ on the left and recalling that $\bb(t)=\W(t)w(t)$, we have that $\|\bb(t)\|^{2}=\|\w(t)\|^{4}$ and so $\|\bb(t)\|=\|\w(t)\|^{2}.$ 
\\ \\Analogously, multiplying \eqref{eq: Matrix-eqq} by $\W(t)$ on the left and by $\W^{\top}(t)$ on the right, we deduce that $(\W(t)\W^{\top}(t))^{2}=\bb(t)\bb^{\top}(t)$. Since both $\W(t)\W^{\top}(t)$ and $\bb(t)\bb^{\top}(t)$ are symmetric positive semi-definite, equality $(\W(t)\W^{\top}(t))^{2}=\bb(t)\bb^{\top}(t)$, implies
\begin{equation*}\W(t)\W^{\top}(t)=\frac{\bb(t)\bb^{\top}(t)}{\|\bb(t)\|}.\end{equation*}
Equation~\eqref{eq: UVector-Flow v-WNN} yields
\begin{align}
\dot{\bb}(t)=\dot{\W}(t)\w(t)+\W(t)\dot{\w}(t)=-\nabla \mathcal{L}(\bb(t)) \|\w(t)\|^{2}- \W(t)\W^{\top}(t)\nabla\mathcal{L}(\bb(t)).\label{eq: Matrix-eq33}
     \end{align}
Plugging in \eqref{eq: Matrix-eq33} the expressions of $\|\w(t)\|^{2}$
 and $\W(t)\W^{\top}(t)$ in terms of $\bb(t)$, we get
\begin{align}
\dot{\bb}(t)=&-\left(\|\bb(t)\|\Id+\frac{\bb(t)\bb^{\top}(t)}{\|\bb(t)\|} \right)\nabla\mathcal{L}(\bb(t)).
     \label{eq: Uvanot0 Flow11}
     \end{align}
Now notice that the Hessian of $F$ and its inverse are given by
\begin{equation*}
\nabla^2F(\bb)=\|\bb\|^{-1/2}\left(\Id-\frac{\bb\bb^{\top}}{2\|\bb\|^{2}}\right)\hspace{5mm}\text{ and }\hspace{5mm} \left[\nabla^2F(\bb)\right]^{-1}=\|\bb\|^{1/2}\left(\Id+\frac{\bb\bb^{\top}}{\|\bb\|^{2}}\right).
\end{equation*} 
Then, multiplying \eqref{eq: Uvanot0 Flow11} on the left by $\left[\nabla^2F(\bb)\right]$, we obtain the following time-warped dynamical system:
\begin{align} \label{eq: Uvanot0 Flow H} 
   \nabla^{2}\left(\frac{2}{3}{\|\bb(t)\|^{3/2}}\right)\dot{\bb}(t)=&-\|\bb(t)\|^{1/2}\nabla\mathcal{L}(\bb(t)),
\end{align}
which is indeed \eqref{eq: x(t)flow} with $F$ and $G$.\hfill
\end{proof}
\begin{remark}[Connection with intrinsic training dynamics]
The preceding proof illustrates the role of conservation laws
in obtaining intrinsic dynamics, as studied systematically in
\cite{marcotte2025intrinsic}.
The conserved quantity
\(
    \W(t)^\top\W(t)-\w(t)\w(t)^\top
\)
and the balanced initialization allow us to express the
induced dynamics solely in terms of $\bb=\W\w$:
\[
    \dot{\bb}=-K(\bb)\nabla\mathcal{L}(\bb),
    \qquad
    K(\bb)=\|\bb\|\Id+
           \frac{\bb\bb^\top}{\|\bb\|}.
\]
An intrinsic representation does not, by itself, guarantee
a mirror-flow representation.
Here, for $\bb\neq 0$, the additional identity
\(
    K(\bb)=G(\bb)\left[\nabla^2F(\bb)\right]^{-1}
\)
provides the time-warped mirror structure needed to apply
our well-posedness, convergence, and implicit-bias results.
\end{remark}
\begin{corollary}\label{cor:7.2}
    Suppose Assumption~\ref{A: D1} and \ref{A: D3} hold with $b=0$. Let $(\W(t),\w(t))$ be the solution of \eqref{eq: UVector-Flow v-WNN} with $\W^{\top}_0\W_0=\w_0\w^{\top}_0$
and let $\bb(t)=\W(t)\w(t)$. Then, for every $z\in\mathcal{S}$ and for every $t > 0$,
\begin{align*}
  \LLL(\bb(t))-\LLL^{*} \quad \leq \quad  \frac{D_{F}(z,\bb_0)}{\sqrt{r} t },
\end{align*}
where $r:=\frac{\mathcal{L}(0)-\mathcal{L}(\bb_0)}{2 \|\nabla \mathcal{L}(0)\|}$. Moreover, $\lim\limits_{t\rightarrow +\infty}\bb(t)=\bb_{\infty}$, where
\begin{align}
\label{implb} 
\left\{\bb_{\infty}\right\} \ = \ \argmin\limits_{z\in\mathcal{S}}  \quad \left\{ D_{F}(z,\bb_0)+\int_{0}^{+\infty}{\|\bb(t)\|^{1/2}}\scal{\nabla\mathcal{L}(\bb(t))}{z-\bb(t)}dt \right\}. 
\end{align}

\end{corollary}
\begin{proof}
    For $b=0$, Assumption \ref{A: D2} is satisfied. In particular, it holds with $C(R)=\sqrt{R}$. Next we show that assumption \ref{A: D4} holds, namely that there exists $\bar{z}\in\mathcal{S}$ such that $D_{F}(\bar{z},\cdot)$ is coercive. Indeed, the Bregman divergence $D_{F}(z,\cdot)$ is coercive for every $z\in\RR^p$: from the Cauchy-Schwarz inequality, 
\begin{align*}
D_{F}(z,\bb)=&\hspace{1mm}\frac{2}{3}\|z\|^{3/2}-\frac{2}{3}\|\bb\|^{3/2}-\|\bb\|^{-1/2}\scal{\bb}{z-\bb}\nonumber\\
% =&\hspace{1mm} \frac{2}{3}\|z\|^{3/2}+\frac{1}{3}\|\bb\|^{3/2}-\|\bb\|^{-1/2}\scal{\bb}{z} \\ 
\geq &\hspace{1mm}\frac{2}{3}{\|z\|^{3/2}}+\frac{1}{3}\|\bb\|^{3/2}-{\|\bb\|}^{1/2}\|z\|\\
 = &\hspace{1mm}\frac{2}{3}{\|z\|^{3/2}}+{\|\bb\|^{1/2}}\left(\frac{1}{3}\|\bb\|-\|z\|\right).
\end{align*}
Then, when the norm of $\bb$ tends to infinity, $D_{F}(z,\bb)$ also tends to infinity. The statement follows from Theorem \ref{T: Existence} and Theorem \ref{T: Main result}.\hfill
\end{proof}
\begin{remark}
    As noted in \cite{azulay2021implicit},  there does not exist an entropy function $F$ such that the dynamical system \ref{eq: Uvanot0 Flow11} can be written as a mirror flow (see Lemma \ref{poincare}), and the results in \cite[Theorem 4.6]{li2022implicit} cannot be applied.
\end{remark}

\begin{remark}
When $\LLL$ is the square loss, this setting has been considered in \cite{azulay2021implicit}. 
Our result generalizes the work in \cite{azulay2021implicit}, where the authors compute the implicit bias for the square loss. Moreover, we prove the existence, uniqueness and convergence of the trajectory to a solution, while in \cite{azulay2021implicit} the authors assume these properties to hold.
\end{remark}

\begin{remark} Our results can be contrasted to those in \cite{arora2018optimization}. In the latter paper, given $N$ matrices $W_i\in\mathbb{R}^{p_i}\times\mathbb{R}^{p_{i+1}}$ for every $i$ with $p_1=p$, $p_{N+1}=1$, the authors consider the reparameterization $\bb=W_1\ldots W_N$. For the special case $N=2$, this setting coincides with ours. In \cite[Claim 2]{arora2018optimization}, the flow \eqref{eq: Uvanot0 Flow11} has been derived, and interpreted in a different way. Our results exploit the to time warping mirror flow interpretation. Thanks to this, we can transform the impossibility result \cite[Theorem 2]{arora2018optimization}, stating that the trajectory $\bb(t)$ does not follow any gradient flow, into an existence result if the geometry is modified through a time warped mirror function. Moreover, the mirror flow interpretation gives the possibility of naturally deriving an implicit bias. 
\end{remark}
\begin{remark}
The case where $\W^{\top}(0)\W(0)-\w(0)\w^{\top}(0)=\gamma\Id$ for $\gamma\neq 0$ could also be considered, as in \cite{azulay2021implicit}. In this case, we could compute the implicit bias, which is the generalization of the one proposed in \cite{arora2018optimization}\cite[Theorem 2]{azulay2021implicit} for the least square, see \cite[Appendix B]{azulay2021implicit}. The computations are very technical and are beyond the scope of this paper. 
\end{remark}
We can further characterize the limit point for fully connected linear networks of depth 2, in the linear setting of Corollary \ref{Remark: Saving energy}. 
\begin{corollary}\label{L: Dsol-Uv}
Suppose that assumptions \ref{A: D1} and \ref{A: D3} are satisfied, assume that $\LLL(\bb)=\ell(X\bb)$ and $\mathcal{S}=\left\{\bb\in\RR^{p}\mid X\bb=y\right\}$, where $X\in \RR^{d\times p}$ and $y\in\RR^{d}$, and let $P: \ \RR^p \to \RR^p$ be the projection operator onto the kernel of $X$. Let $(\W(t),\w(t))$ be the solution of \eqref{eq: UVector-Flow v-WNN} with
$\W^{\top}_0\W_0=\w_0\w^{\top}_0$
and let $\bb(t)=\W(t)\w(t)$. Then $\bb(t)$ satisfies that 
\begin{align}
\frac{d}{dt}\left(\frac{P(\bb(t))}{\|\bb(t)\|^{1/2}}\right)&=0
\nonumber
\end{align}
and that $\lim\limits_{t\to +\infty} \bb(t)=\bb_{\infty}$ with
\begin{align*}
\{ \bb_{\infty}\} \ = \ &\argmin\limits_{z\in\mathcal{S}} \left\{ \frac{2}{3}\|z\|^{3/2}-\frac{\scal{\bb_{0}}{z}}{\sqrt{\|\bb_{0}\|}} \right\}. 
\end{align*}
Moreover, let
\begin{align}
    \{ \bb_{*} \} = \argmin_{z\in\mathcal{S}}\ \|z\|^2.
    \label{P: minimal norm Uv}
\end{align}
Then,
\begin{align}\label{eq:LSreparam}
    \bb_{\infty}= \ \bb_{*}+ \alpha P(\bb_{0}), \text{\quad where \quad}\alpha^2=&\frac{\|P(\bb_{0})\|^2+\sqrt{\|P(\bb_{0})\|^4+4\|\bb_{0}\|^2\|\bb_{*}\|^2}}{2\|\bb_{0}\|^2}.
\end{align}   

\end{corollary}
\noindent The proof is given in the Appendix; see Subsection \ref{P: Dsol-Uv}.
The previous result establishes that when we run gradient flow on $\theta=(\W,\w)$, the corresponding trajectory is biased towards a specific element of $\mathcal{S}$, that is given by the sum of two terms: $\bb_*$, the minimal-norm vector in $\mathcal{S}$; and a scaling of $P(\bb_0)$, the projection of the initialization onto the kernel of $X$. 
\begin{remark}
We compare the previous result with the well-known one for vanilla gradient flow, namely
\begin{equation} \label{eq:vgf}
    \begin{cases}
        \dot{\tilde{\bb}}=-\nabla \mathcal{L}(\tilde{\bb}(t))\\
        \tilde{\bb}(0)=\bb_0.
    \end{cases}
\end{equation}
On the one hand, the system~\eqref{eq:vgf} is an instance of \eqref{eq: x(t)flow} with $F(\bb)=(1/2)\norm{\bb}^2$ and $G(\bb)=1$. Therefore, $D_F(z,\bb_0)=(1/2)\norm{z-\bb_0}^2$, and we derive from Corollary \ref{Remark: Saving energy} that $\tilde{\bb}(t) \to \tilde{\bb}_{\infty}$, where 
$\tilde{\bb}_{\infty}$ is the projection onto $\mathcal{S}$ of $\bb_0$, and so
\begin{equation}\label{eq:LS}
\tilde{\bb}_{\infty}=\bb_*+P(\bb_0).
\end{equation}
Since $\bb_{\infty}$ and $\tilde{\bb}_{\infty}$ lie in $\mathcal{S}$, they can be written as $\bb_*$ plus an element in $\mathrm{ker} X$. Equations~\ref{eq:LSreparam} and \ref{eq:LS} show that they  both lie on  the line passing through $\bb_*$ and having direction $P(\bb_0)$. Recall that, under our assumptions $\mathcal{L}(\beta_0)<\mathcal{L}(0)$, thus $\beta_0\neq 0$. However, $\bb_{\infty}=\tilde{\bb}_{\infty}$ if $\bb_0\in\ker(X)^\perp$. Otherwise, the main difference in the implicit bias between running vanilla gradient flow on $\bb$ and gradient flow on the reparameterization $\bb=\W\w$ is given by the scaling factor of $P(\bb_0)$, namely $\alpha$. Also note that when $\|\beta_0\|\to 0$, both $\beta_\infty$ e $\tilde{\beta}_\infty$ tend to $\beta_*$, but with a different speed, since $\alpha$ goes to $+\infty$. 
\end{remark}
\section{Application to weight normalization}\label{Sec: WN}
In this section we consider reparameterizations in terms of polar coordinates, related to weight normalization \cite{wu2020implicit,morwani2022inductive,salimans2016weight}, namely of the form 
$\beta=q(\theta)$, with $\theta=(\alpha,w)\in\mathbb{R}\times\mathbb{R}^p$ where
\begin{equation}
q(\alpha,w)=h(\alpha)\frac{w}{\|w\|},
    \label{reparam_norm}
\end{equation}

for a given $h\colon\mathbb{R}\to\mathbb{R}$. The associated gradient flow on the reparameterization is given by:
\begin{align}
     \Dot{\w}(t)=&-\nabla_{\w}\mathcal{L}(\bb(t))=-\frac{h(\alpha(t))}{\|\w(t)\|}\left(\Id-\frac{\w(t)\w^{\top}(t)}{\|\w(t)\|^{2}}\right)\nabla \mathcal{L}(\bb(t)) \label{eq: FWnormalvector-Flow-u}\\ \Dot{\alpha}(t)=&-\nabla_{\alpha}\mathcal{L}(\bb(t))=-h'(\alpha(t))\scal{\frac{\w(t)}{\|\w(t)\|}}{\nabla \mathcal{L}(\bb(t))}.
     \label{eq: FWnormalvector-Flow-alpha}
\end{align}
In the following, we will show that this flow is equivalent to a generalized mirror one, for a mirror function of the form $F(\bb)=f(\|\bb\|)$ and a time warping of the form $G(\bb)=g(\|\bb\|)$.  
To derive the main result of this section, we need an auxiliary lemma, providing a reformulation of \eqref{eq: x(t)flow} for radial functions. 
From now on, unless explicitly stated otherwise, when we refer to assumptions \ref{A: D2} and \ref{A: D1}, we mean them with $b=0$. \vspace{2mm}
\begin{lemma}
\label{L: Hessian-Norm} Let $\mathcal{L}\colon \R^p \to \R$ be differentiable. Let $f\colon [0,+\infty)\to[0,+\infty)$ be a twice differentiable function on $(0,+\infty)$ such that $f'>0$ and $f''>0$ and let $g\colon [0,+\infty)\to[0,+\infty)$. Consider the following dynamical system:
\begin{align}
\label{eq: MF-Norm}
\begin{cases}
\Dot{\bb}(t)&=-\dfrac{g(\|\bb(t)\|)\|\bb(t)\|}{f'(\|\bb(t)\|)}\left(\Id -\dfrac{\bb(t)\bb^{\top}(t)}{\|\bb(t)\|^{2}}\right) \nabla\mathcal{L}(\bb(t))-\dfrac{g(\|\bb(t)\|)}{f''(\|\bb(t)\|)}\dfrac{\bb(t)\bb^{\top}(t)}{\|\bb(t)\|^{2}}\nabla\mathcal{L}(\bb(t))\\
\bb(0)&=\bb_0.
\end{cases}
\end{align}
Then the latter is equivalent to \eqref{eq: x(t)flow} with $F(\bb)=f(\|\bb\|)$ and $G(\bb)=g(\|\bb\|)$. Moreover, assume that $1/f^{'}$ and $1/f^{''}$ are locally Lipschitz on $(0,+\infty)$ and $g$ is locally Lipschitz on $(0,+\infty)$ and such that, for every $r>0$, there exists $c$ such that $g(s)>c$ for every $s>r$. Then assumption~\ref{A: D2} holds.
\end{lemma}\hspace{-6mm} 
\noindent The proof is given in the Appendix; see Subsection \ref{Proof: Hessian-Norm}.\vspace{2mm}
\begin{remark}
 By \eqref{eq: MF-Norm}, we get that
 \begin{align*}
 \frac{d}{dt} \frac{\bb(t)}{\|\bb(t)\|}=& \dfrac{1}{\|\bb(t)\|}\left(\Id -\dfrac{\bb(t)\bb^{\top}(t)}{\|\bb(t)\|^{2}}\right)\Dot{\bb}(t)=-\dfrac{g(\|\bb(t)\|)}{f'(\|\bb(t)\|)}\left(\Id -\dfrac{\bb(t)\bb^{\top}(t)}{\|\bb(t)\|^{2}}\right) \nabla\mathcal{L}(\bb(t)).\\
   \frac{1}{2}\frac{d}{dt} \|\bb(t)\|^2=&\scal{\bb(t)}{\dot{\bb}(t)}=- \frac{g(\|\bb(t)\|)}{f''(\|\bb(t)\|)}\scal{\bb(t)}{\nabla\mathcal{L}(\bb(t))}\nonumber
 \end{align*}
 This computation shows that the two terms on the right-hand side  of equation~\ref{eq: MF-Norm} can  be related separately to the derivative of intensity and direction of $\beta(t)$. In particular, if the second term on the right-hand side of \eqref{eq: MF-Norm} is zero, then the norm is constant.  
\end{remark}

Next, we  show that the gradient flow in \eqref{eq: FWnormalvector-Flow-u}-\eqref{eq: FWnormalvector-Flow-alpha} corresponds to a (radial) generalized mirror flow of the form \eqref{eq: MF-Norm}. 

\begin{theorem}\label{T: dynamics of x(t) Fwn}
Let assumption \ref{A: D1} hold. Consider the dynamics on $(\alpha,w)$ defined in \eqref{eq: FWnormalvector-Flow-u}-\eqref{eq: FWnormalvector-Flow-alpha} with initialization $(\alpha_0,w_0)$ such that $h(\alpha_0)>0$ and $\|\w_0\|=\gamma_0>0$ and assume that assumption \ref{A: D3} holds for $\bb_0=h(\alpha_0)w_0/\|w_0\|$.
Assume that there exists a function $f\colon [0,+\infty)\to[0,+\infty)$ twice differentiable on $(0,+\infty)$ satisfying the following equality for every $\alpha >0$
\begin{align}
(h'(\alpha))^2=\frac{f'(h(\alpha))h(\alpha)}{\gamma_0^2 f''(h(\alpha))}.
\label{A5: f-g-h assumption radial}
\end{align}
Then $\bb=h(\alpha)\frac{\w}{\|\w\|}$ satisfies
\begin{align}\Dot{\bb}(t)=&-\frac{\|\bb(t)\|^{2}}{\gamma_0^2}\left(\Id-\frac{\bb(t)\bb^{\top}(t)}{\|\bb(t)\|^{2}}\right)\nabla\mathcal{L}(\bb(t)) - \frac{f'(\|\bb(t)\|)\|\bb(t)\|}{\gamma_0^2 f''(\|\bb(t)\|)}\frac{\bb(t)\bb^{\top}(t)}{\|\bb(t)\|^2}\nabla\mathcal{L}(\bb(t)).\nonumber\end{align}
Moreover, assume that $f'>0$ and $f''>0$ on $(0,+\infty)$ and that $1/f'$ and $1/f''$ are locally Lipschitz on $(0,+\infty)$. Suppose in addition that for every $r>0$ there exists $c>0$ such that $f'(s)>c/s$ for every $s>r$. Define $g\colon \ \RR_{+}\to\RR$ as 
\begin{align}
g(s):=\frac{s f'(s)}{\gamma_0^2}.
\end{align}
Then, for every $t>0$, $\|\w(t)\|=\gamma_0$ and $h(\alpha(t))\geq \tilde{r}:= \frac{\mathcal{L}(0)-\mathcal{L}(\bb_0)}{2 \|\nabla \mathcal{L}(0)\|}$. Moreover, the dynamics of $\bb(t)$ satisfies \eqref{eq: x(t)flow} with $F(\bb)=f(\|\bb\|)$ and $G(\bb)=g(\|\bb\|)$
and assumption \ref{A: D2} holds. Finally, if $\lim_{s\to+\infty}f'(s)=+\infty$, then Assumption~\ref{A: D4} holds.
\end{theorem}

\noindent The proof is given in the Appendix; see Subsection \ref{Proof: dynamics of x(t) Fwn}.
\begin{remark}
\begin{enumerate}
\item The previous theorem guarantees that, if the initialization satisfies $h(\alpha_0)>0$, then $h(\alpha(t))\geq \tilde{r}>0$ for every $t>0$. Hence, from the definition of weight normalization reparameterization $\bb=h(\alpha)\frac{\w}{\|\w\|}$, we have that $\|\bb(t)\|=h(\alpha(t))$ for every $t>0$.
\item The previous theorem gives a sufficient condition on the function $h$ to cast the gradient flow over $\alpha$ and $\w$ into a time-warping mirror flow on $\bb$, as in Theorems \ref{T: Existence} and \ref{T: Main result}. These results allow us to establish the existence and uniqueness of the trajectory $\bb$ and convergence to the implicit bias. 
\item Note that not all choices of $h$ are viable. For instance, if $h(\alpha)=1$ for every $\alpha$, it is not possible do find a function $f$ satisfying the assumptions of our theorem. 
\item If $h'(\alpha)$ can be expressed in terms of $h(\alpha)$ as $h'(\alpha)=\varphi(h(\alpha))$ for some function $\varphi: \ \R_+\to \R_+$, we can determine $f'$ in \eqref{A5: f-g-h assumption radial} by solving the following differential equation:
\begin{equation*}\gamma_0^2 \left(\varphi(s)\right)^2 f''(s)-sf'(s)=0, \quad \quad s>0.\end{equation*}
\end{enumerate}
\end{remark}
In the context of Corollary \ref{Remark: Saving energy}, where the loss is a composition of a convex function with a linear operator and the solution set is a linear subspace, the next corollary allows us to further characterize the limit point of the flow deriving from weight normalization. Similarly to Corollary \ref{L: Dsol-Uv}, the limit point can be expressed as the sum of the minimal norm solution and a term in the kernel of the linear operator. We generalize the techniques used in \cite[Theorem 2.6]{wu2020implicit} to a broader class of weight normalization reparameterizations.
\begin{corollary}\label{L: CWNS} Under the same assumptions of Theorem~\ref{T: dynamics of x(t) Fwn}, suppose that $\LLL(\bb)=\ell(X\bb)$ and that $\mathcal{S}=\left\{\bb\in\RR^{p}\mid X\bb=y\right\}$, where $X\in \RR^{d\times p}$, $\ell: \RR^d \to \RR$ and $y\in\RR^{d}$. Additionally, suppose that assumption~\ref{A: D4} is satisfied. 
Moreover, let $P: \ \RR^p \to \RR^p$ be the projection operator onto the kernel of $X$. Then, $\beta(t)$ satisfies
\begin{align}
\frac{d}{dt}\left(\frac{P(\bb(t))f'(\|\bb(t)\|)}{\|\bb(t)\|}\right)=0.
\nonumber
\end{align}
and for $t\to +\infty$, $\bb(t)\to \bb_{\infty}$ with
\[\{\beta_\infty\}=\argmin_{z\in\mathcal{S}} \left\{f(\|z\|)-\frac{f'(\|\bb_0\|)}{\|\bb_0\|}\langle\bb_0,z\rangle\right\}.
\]
Moreover, let $\bb_{*}$ be the (unique) solution of $\min_{z\in\mathcal{S}}\ \|z\|^2$.
\begin{align}
\bb_{\infty}=\bb_{*}+P(\bb_{0})\frac{f'(\|\bb_{0}\|)\|\bb_{\infty}\|} {f'(\|\bb_{\infty}\|)\|\bb_{0}\|}.\nonumber
\end{align}

\end{corollary}

\noindent The proof is given in the Appendix in Subsection \ref{P: CWNS}.
Note that, compared to the analogous result in Corollary \ref{L: Dsol-Uv}, this is a fixed point equation, in the sense that we cannot compute the norm of $\bb_{\infty}$ explicitly.

\section{Weight normalization: Examples}
In the following two examples, we explore the flexibility of our formulation. By choosing appropriately $h$ and $f$, we obtain the two well-known reparameterizations corresponding to polynomial and exponential weight normalization \cite{wu2020implicit,morwani2022inductive,salimans2016weight}. Moreover, we consider another reparametrization, where we combine polar coordinates weight normalization for fully connected networks.
\subsection{Polynomial Weight normalization}\label{R: WN}  Let Assumption \ref{A: D1} be satisfied and consider a reparameterization as in \eqref{reparam_norm} with 
$  h(\alpha)=\alpha^{L} .  
 $ 
The dynamics defined in \eqref{eq: FWnormalvector-Flow-u}-\eqref{eq: FWnormalvector-Flow-alpha} become
\begin{align}
     \Dot{\w}(t)=&-\frac{\alpha(t)^L}{\|\w(t)\|}\left(\Id-\frac{\w(t)\w^{\top}(t)}{\|\w(t)\|^{2}}\right)\nabla \mathcal{L}(\bb(t)) \label{eq: FWnormalvector-Flow-u-poly}\\ \Dot{\alpha}(t)=&-L\alpha(t)^{L-1}\scal{\frac{\w(t)}{\|\w(t)\||}}{\nabla \mathcal{L}(\bb(t))}.
     \label{eq: FWnormalvector-Flow-alpha-poly}
\end{align}

\begin{proposition}\label{prop:radialpoly}
Let assumption \ref{A: D1} hold. Consider the dynamics on $(\alpha,w)$ defined in \eqref{eq: FWnormalvector-Flow-u-poly}-\eqref{eq: FWnormalvector-Flow-alpha-poly} with initialization $(\alpha_0,w_0)$ such that $\alpha_0^L>0$ and $\|\w_0\|=\gamma_0>0$ and assume that assumption \ref{A: D3} holds for $\bb_0=\alpha_0^Lw_0/\|w_0\|$ and $b=0$. Then the differential equation in the original variable $\bb(t)=\alpha(t)^L w(t)/\|w(t)\|$ is
\begin{align}
\label{eq: eqqq}
\Dot{\bb}(t)&=-\|\bb(t)\|^2\left(\Id -\dfrac{\bb(t)\bb^{\top}(t)}{\|\bb(t)\|^{2}}\right) \nabla\mathcal{L}(\bb(t))-L^2\dfrac{\bb(t)\bb^{\top}(t)}{\|\bb(t)\|^{2/L}}\nabla\mathcal{L}(\bb(t)), \quad\quad  t>0.
\end{align}
This equation has a unique solution $\bb(t)$ defined on $[0,+\infty)$. Moreover, the results of Theorem~\ref{T: Main result} hold.
\end{proposition}
\begin{proof}
 Let $f\colon[0,+\infty)$ be such that
\begin{equation*}
  f'(s)=\exp\left(\frac{s^{2/L}}{2L\gamma_0^2}\right) \text{for every } s>0
\end{equation*} so that equation \eqref{A5: f-g-h assumption radial} is satisfied. 
Note that $f$ is twice differentiable on $(0,+\infty)$,  $f'(s)>0$, $f''(s)>0$ for all $s>0$, and $1/f''$ is locally Lipschitz. Let $r>0$. Then we have that $f'(s)>c/s$ for every $s>r$ with $c:=r\exp\left(\frac{r^{2/L}}{2L\gamma_0^2}\right)$.
Define $g(s)=sf'(s)$ for every $s\geq 0$. Then, Theorem \ref{T: dynamics of x(t) Fwn} implies that Assumptions~\ref{A: D2} is satisfied. Since $\lim_{s\to+\infty}f'(s)=+\infty$, also Assumption~\ref{A: D4} holds so we can apply Theorem \ref{T: Existence} and Theorem \ref{T: Main result}.  \hfill
\end{proof}

As we showed in the proof of the previous Proposition, Assumption~\ref{A: D4} is satisfied in this setting, so that we can derive an implicit bias by specializing the statement of Corollary~\ref{L: CWNS}. 

\paragraph{Implicit bias for $L=1$}
 The choice $L=1$, which corresponds to the weight normalization proposed by \cite{salimans2016weight}, in the setting of Corollary \ref{L: CWNS}, we obtain
\begin{align}
\bb_{\infty}=\bb_{*}+P(\bb_{0})\frac{\|\bb_{\infty}\|}{\|\bb_{0}\|}e^{\frac{\|\bb_{0}\|^2-\|\bb_{\infty}\|^{2}}{2\gamma_0^2}},
    \label{F: EFWN1}
\end{align}
where we recall that $P$ is the projection operator on the kernel of $X$ and that $\|\beta_0\|=\alpha_0$ . Note that when $\bb_{0}\in\ran(X^{\top})=\ker(X)^{\perp}$, then $\bb_{\infty}$ is equal to the minimal norm solution $\bb_{*}$. Our results guarantee that the trajectory exists for $t\to+\infty$, which is something usually assumed, and not proved.
The implicit bias for this example has been studied in \cite{wu2020implicit} and \cite{morwani2022inductive} for the least squares and the logistic loss, respectively. In \cite{wu2020implicit},  the convergence to a stationary point (zero loss or zero norm ) is analyzed in the continuous and discrete settings, and  the expression in \eqref{F: EFWN1} is derived. 

\paragraph{Implicit bias for $L=2$}
 The case $L=2$ is called squared weight normalization. 
 In the setting of Proposition~\ref{prop:radialpoly} we can explicitly compute the functions $f$ and $g$ since in this case 
$f'(s)=\exp{\frac{s}{4\gamma_0^2}}$.
 In the setting of Corollary \ref{L: CWNS}, we get that 
\begin{align*}
\bb_{\infty}=\bb_{*}+P(\bb_{0})\frac{\|\bb_{\infty}\|}{\|\bb_{0}\|}e^{\frac{\|\bb_{0}\|-\|\bb_{\infty}\|}{4\gamma_0^2}}.
\end{align*}  

\begin{remark}
    Note that in this case, to determine the implicit bis, not only the choice of $\beta_0$ matters, but also that of $\alpha_0$ and $w_0$. 
\end{remark}

\subsection{Exponential weight normalization.} \label{expweight} 
In this subsection we consider the case
$
h(\alpha)=e^{\frac{\alpha}{\tau}}
$
for some $\tau>0$. The following proposition holds.
\begin{proposition}\label{prop:radialpolyI}
 Let assumption \ref{A: D1} hold. Consider the dynamics on $(\alpha,w)$ defined in \eqref{eq: FWnormalvector-Flow-u}-\eqref{eq: FWnormalvector-Flow-alpha} with initialization $(\alpha_0,w_0)$ such that $\|\w_0\|=\gamma_0>0$ and assume that assumption \ref{A: D3} holds for $\bb_0=e^{\alpha_0/\tau}w_0/\|w_0\|$ and $b=0$. Then the differential equation in the original variable $\bb(t)=e^{\alpha(t)/\tau} w(t)/\|w(t)\|$ is
\begin{align}\label{eq:punto}
\Dot{\bb}(t)&=-\left[\frac{\|\bb(t)\|^2}{\gamma_0^2}\Id -\frac{\tau^2-\gamma_0^2}{\tau^2\gamma_0^2}{\bb(t)\bb^{\top}(t)}\right] \nabla\mathcal{L}(\bb(t)), \quad\quad  t>0.
\end{align} 
This equation has a unique solution $\bb(t)$ defined on $[0,+\infty)$, and the results of Theorem~\ref{T: Main result} hold.
\end{proposition}
\begin{proof}
By solving equation \eqref{A5: f-g-h assumption radial}, we derive
\begin{equation*}
f(s)=\frac{\gamma_0^2}{\tau^{2}+\gamma_0^2}s^{(\tau^{2}+\gamma_0^2)/\gamma_0^2}, \quad \text{and}\quad g(s)=\frac{1}{\gamma_0^2}s^{(\tau^{2}+\gamma_0^2)/\gamma_0^2}.
\end{equation*}
On the interval $(0,+\infty)$, we have $f'>0$, $f''>0$,  $1/f'$ and $1/f''$ locally Lipschitz. In addition, $g(s)>r^{(\tau^2+\gamma_0^2)/\gamma_0^2}$ for every $s>r$. Then it follows from Theorem \ref{T: dynamics of x(t) Fwn} that $\bb(t)=h(\alpha(t))\frac{w(t)}{\norm{w(t)}}$ solves the differential equation \eqref{eq:punto}.
To apply Theorem \ref{T: Existence}, we note that assumption \ref{A: D4} is satisfied since $\lim_{s\to+\infty} f'(s)=+\infty$.\hfill \end{proof}

In the setting of Corollary \ref{L: CWNS}, we obtain that
\begin{align*}    \bb_{\infty}=\bb_{*}+P(\bb_{0})\left(\frac{\|\bb_{0}\|}{\|\bb_{\infty}\|}\right)^{(\tau^2-\gamma_0^2)/\gamma_0^2}.
\end{align*}

\begin{remark} When $(\tau^2+\gamma_0^2)/\gamma_0^2=3/2$, the mirror function $F$ is the same function obtained in Example \ref{Ex: MV}, showing that different reparameterizations may have the same implicit bias. Notice that the  dynamics for $\bb$ are different, since the time warping $G$ is not the same in the two cases.
\end{remark}

\subsection{Weight normalization of a fully connected network}\label{Sec: app WN2}In this section, we combine the two main examples presented in this article: fully connected linear networks and weight normalization.

 We consider a reparameterization involving a decomposition of the vector $\bb$ into the product of three terms: a scalar that represents the norm of $\bb$, and the product of a matrix by a vector, product that represents the direction of $\bb$ with unitary norm:
\begin{align}\label{eq:fullypolar} \bb=h(\alpha)\frac{\W\w}{\|\W\w\|},
\end{align}
where $h\colon \R\to\RR_{+}$ is a given function, while $\alpha\in\R$, $\W\in\RR^{p\times m}$ and $\w\in\RR^{m}$ are the new variables. The
associated gradient flow on the reparameterization is given by the following dynamical system:
\begin{align}
    \Dot{\alpha}(t)=&-\nabla_{\alpha}\mathcal{L}(\bb(t))=-h'(\alpha(t))\scal{\frac{\bb(t)}{\|\bb(t)\|}}{\nabla\mathcal{L}(\bb(t))}
     \label{eq: WMatrixVector-Flow-a}\\
      \Dot{\W}(t)=&-\nabla_{\W}\mathcal{L}(\bb(t))=-\frac{h(\alpha(t))}{\|\W(t)\w(t)\|}\left(\Id-\frac{\bb(t)\bb^{\top}(t)}{\|\bb(t)\|^{2}}\right)\nabla\mathcal{L}(\bb(t))\w^{\top}(t)\label{eq: WMatrixVector-Flow-U}\\
     \Dot{\w}(t)=&-\nabla_{\w}\mathcal{L}(\bb(t))=-\frac{h(\alpha(t))}{\|\W(t)\w(t)\|}\W^{\top}(t)\left(\Id-\frac{\bb(t)\bb^{\top}(t)}{\|\bb(t)\|^{2}}\right)\nabla\mathcal{L}(\bb(t)).\label{eq: WMatrixVector-Flow-v}
    \end{align}
 We consider, for simplicity of computations, an initialization $(\alpha_0,W_0,w_0)$ such that $\|\W_0\w_0\|=1$ and $\w_0\w_0^{\top}=\W_0^{\top}\W_0$.
 In the next theorem, we derive a time warping mirror flow on $\bb$ as in \eqref{eq: x(t)flow} from the gradient flow defined in \eqref{eq: WMatrixVector-Flow-a}-\eqref{eq: WMatrixVector-Flow-v} over $\alpha$, $\W$, and $\w$, respectively. We then apply Theorems \ref{T: Existence} and \ref{T: Main result}, providing the existence and uniqueness of solution for the dynamical system and convergence in value and in trajectory. In the case where the loss is a composition of a convex function and a linear operator, the limit point can be characterized as the Bregman projection on the solution set for a given mirror map, which depends on the choice of the reparameterization. 
\begin{theorem}\label{T: dynamics of x(t) aUV}
Let assumption \ref{A: D1} hold. Consider the overparameterized model defined in \eqref{eq:fullypolar} and the dynamic defined in \eqref{eq: WMatrixVector-Flow-a}-\eqref{eq: WMatrixVector-Flow-v} with initialization such that $\|\W_0\w_0\|=1$ and $\w_0\w_0^{\top}=\W_0^{\top}\W_0$ and assume that assumption \ref{A: D3} holds for $\bb_0=h(\alpha_0)(W_0w_0)/\|W_0w_0\|$.
Assume that there exists a function $f\colon [0,+\infty)\to[0,+\infty)$ twice differentiable on $(0,+\infty)$ such that $f''(s)>0$ for every $s>0$, $1/f''$ is locally Lipschitz on $(0,+\infty)$ and satisfying the following equality
\begin{align}
(h'(\alpha))^2=\frac{f'(h(\alpha))h(\alpha)}{f''(h(\alpha))}.
\label{A5: f-g-h assumption radial1}
\end{align}
Suppose in addition that for every $r>0$ there exists $c>0$ such that $f'(s)>c/s$ for every $s>r$  and 
define $g\colon \ \RR_{+}\to\RR$ as $g(s)=s f'(s)$. Then  assumption \ref{A: D2} is satisfied, and for every $t>0$, $\|\W(t)\w(t)\|=1$, $h(\alpha(t))>0$, and $\bb(t)$ satisfies \eqref{eq: x(t)flow}, with $F(\bb)=f(\|\bb\|)$ and $G(\|\bb\|)=g(\|\bb\|)$. Finally, if $\lim_{s\to+\infty} f'(s)=+\infty$, then Assumption~\ref{A: D4} holds.\end{theorem}
\noindent The proof is given in the Appendix in Subsection \ref{P: dynamics of x(t) aUV}.
\begin{remark}
\label{R: WN1} Choosing $h$, $f$, and $g$ as in Examples \ref{R: WN} and \ref{expweight}, the assumptions of Theorem~\ref{T: dynamics of x(t) aUV} are satisfied. Consequently, $\bb(t)$ satisfies the properties derived in Theorem \ref{T: Existence}, Theorem \ref{T: Main result} and Corollary \ref{Remark: Saving energy}.
\end{remark}

\section{Conclusion and Future Work}\label{sec:conclusion}
In this article, we studied a unified framework encompassing several reparemeterization schemes from the existing literature. We provided assumptions that guarantee the well-posedness of the dynamical system and convergence of the trajectory. This convergence is both in functional value and in the variable itself, that tends to a minimizer of the loss function avoiding the additional stationary point caused by the reparameterization. In the case where the function is composed by a linear operator, an explicit expression of the implicit bias is given. This comprehensive approach allows us to gain insights into the behavior of the time warped mirror flow, providing a more comprehensive understanding of its properties.\\ \\We also provide a criterion to determine, for a certain function that depends only on the norm, a suitable weight normalization parameterization. Finally, we apply the obtained results to a general multi-neuron fully connected linear network of depth 2, different schemes of weight normalization, and two different extensions of these two previous models. \\ \\
In the future, we would like to extend this work to models that are more closely related to neural networks and that introduce non-smooth activation functions. Possible extensions are: 1) Extend the results to more general models, such as neural networks; 2) Study new ways to train the overparameterization, for example Stochastic Gradient Descent or second-order methods; 3) Extend this approach to differential inclusions; 4) Study the problem in the matricial case.

\section{Acknowledgement}\label{s: Conclusion}
This project has been supported by the TraDE-OPT project, which received funding from the European Union’s Horizon 2020 research and innovation program under the Marie Skłodowska-Curie grant agreement No 861137. L. R. acknowledges the financial support of the European Research Council (grant SLING 819789), the AFOSR projects FA9550-18-1-7009 (European Office of Aerospace Research and Development), the EU H2020-MSCA-RISE project NoMADS - DLV-777826, and the Center for Brains, Minds and Machines (CBMM), funded by NSF STC award CCF-1231216. S. V. and L. R. acknowledge the support of the AFOSR project FA8655-22-1-7034. The research by C.V., S. V. and C. M. has been supported by the MIUR Excellence Department Project awarded to Dipartimento di Matematica, Universit\'a di Genova, CUP D33C23001110001. C.V. acknowledges the financial support of the Universidad de Tarapac\'a internal research project Fortalecimientos Grupos de Investigaci\'on C\'odigo 8802-25. S. V. and C. M. are members of the Gruppo Nazionale per l’Analisi Matematica, la Probabilità e le loro Applicazioni (GNAMPA) of the Istituto Nazionale di Alta Matematica (INdAM). This work represents only the view of the authors. The European Commission and the other organizations are not responsible for any use that may be made of the information it contains.

\printbibliography
\appendix
\section{Preliminaries}
\label{sec:NB}
We begin by introducing some notation used throughout this paper. \\ \\Matrices are denoted with uppercase letters and vectors with lowercase letters. For any positive integer $p$, we denote by $[p]$ the set of all integers from $1$ to $p$. The Hadamard product of two vectors $a$ and $b$ is represented as $a \astrosun b$, where each component is obtained by multiplying the corresponding components of $a$ and $b$: $(a \astrosun b)_{i} = a_i b_i$. For a positive scalar $L$, the vector $a^{\astrosun L}$ is defined as $\left(a^{L}_{i}\right)_{i=1}^{p}$, where each component is raised to the power of $L$. We denote by $\vec{\alpha}$ the $p$-dimensional vector which all components are equal to $\alpha$. The diagonal matrix $\Diag(a)$ is a square matrix with the components of the vector $a\in\RR^{p}$ on its main diagonal and zeros elsewhere. If $M$ is an invertible $p \times p$ matrix and $u$ and $v$ are two $p$-dimensional vectors such that $v^{\top}M^{-1}u \neq -1$, the Sherman–Morrison formula can be used to calculate $(M + uv^{\top})^{-1}$ as
\begin{align}\label{Sherman}
(M+uv^{\top})^{-1}=M^{-1}-\frac{M^{-1}uv^{\top}M^{-1}}{1+v^{\top}M^{-1}u}.
\end{align}
We denote by $M^{\top}$, $\ran(M)$, and $\ker(M)$, the transpose, the range, and the kernel of $M$, respectively.
A matrix $M\in \RR^{p\times p}$ is called positive definite if and only if $a^{\top}Ma>0$ for every $a\in\RR^{p}\backslash\{0\}$.
\\ \\The Euclidean norm in $\RR^{p}$ is denoted by $\|\cdot\|$.  For a given set $\mathcal{S}$, we denote by $\overline{\mathcal{S}}$, the closure of $\mathcal{S}$ and $\mathcal{S}^{\perp}$ the orthogonal complement if $\mathcal{S}$ is a linear subspace.

The Hessian of a function $F\hspace{1mm}\colon\dom(F)\mapsto \RR$ is denoted by $\nabla^{2}F$ and the Jacobian of $q:\RR^{d}\to \RR^{p}$ at $\theta \in \R^d$ is denoted as $J_{q}(\theta)$, which is a matrix in $\RR^{p\times d}$.

\begin{definition}\cite[Definition 1.17]{barreira2012ordinary} A function $f\colon U\to \RR^{d}$ in a set $U\subset\RR^{p}$ is said to be locally Lipschitz in $U$ if for each compact set $K\subset U$ there exists $L>0$ such that
 \begin{align}
     (\forall x,y\in K)\hspace{5mm}\|f(x)-f(y)\|\leq L\|x-y\|.\nonumber
 \end{align}
    \end{definition}
   
% \begin{lemma}(Descent Lemma) \cite[Lemma 1.30]{peypouquet2015convex}
%   If $f\colon \RR^{p}\mapsto \RR$ is G\^ateaux-differentiable and $\nabla f$ is
% Lipschitz-continuous with constant $L$, then
% \begin{align}
%     f(\bb)\leq f(z)+\scal{\nabla f(z)}{\bb-z}+\frac{L}{2}\|z-\bb\|^{2},\nonumber
% \end{align}
% for each $(\bb,z)\in\RR^{p}\times\RR^{p}.$
% \end{lemma}
\hspace{2mm}

The next Lemma provides a necessary condition for a vector field to be a gradient.
\begin{lemma}\cite[Corollary 10.6]{apostol1969multi}\label{poincare} Let $f\colon U \subseteq \RR^n \to \RR^n$, $f=(f_1,\ldots,f_{n})$ be a continuously differentiable vector field on an open set $U$. If there exists $g\colon U\to\RR$ such that $f=\nabla g$, then the partial derivatives of the components of $f$ are related by the equations:
\begin{align}\centering
  (\forall  i,j \in\{1,\ldots,n\}) (\forall \bb\in U) \qquad \frac{\partial f_{i}}{\partial \bb_{j}}(\bb)= \frac{\partial f_{j}}{\partial \bb_{i}}(\bb).
\end{align}
 \end{lemma}

\section{Examples}\label{ex: UV and MV}
The following examples have been adapted from \cite[Theorem 1]{chou2023more}, \cite[Theorem 1]{arora2018optimization}, respectively. The main objective is to illustrate how to use equation \eqref{eq:MF0} and its limitations. First, in example \ref{Ex: uL}, we study the case when the reparameterization is separable and find the mirror map is trivial. Second, in example \ref{Ex: UV}, we give a scenario where the derivation of the mirror flow is not straightforward and requires a technical effort. 
\begin{example}\label{Ex: uL}
Consider a $L$-layer's overparameterized vector factorization, which can be mathematically expressed as $\bb=q(\theta)=\theta^{\astrosun L}\in\RR^{p}$ with $L\geq 2$. We define the flow over $\theta$ as:
\begin{align}\Dot{\theta}(t)= -\nabla_\theta \cL (q(\theta(t)))=-L\nabla\mathcal{L}(\bb(t))\astrosun \theta^{\astrosun (L-1)}(t),\nonumber\end{align}
where the last equality comes from the chain rule.\\ \\
Note that $q$ satisfies \eqref{eq:MF0} with $F$ equal to:
\begin{align}
    F\colon\hspace{2mm} & \bb\mapsto\left\{\begin{array}{cc}
       \frac{1}{4}\sum\limits_{n=1}^{k}\left(\bb_{n}\log(\bb_{n})-\bb_{n}\right)  &  \text{If }L=2\\
       \frac{1}{2L(2-L)}\sum\limits_{n=1}^{k}\bb^{\frac{2}{L}}_{n}  & \text{If }L>2.
    \end{array}\right., \label{eq:had}
\end{align}
Indeed, if we compute $J_{q}(\theta)J_{q}(\theta)^{T}$, we get:
\begin{align}J_{q}(\theta)J_{q}(\theta)^{T}=\left(q'(\theta)\right)^{\astrosun 2}=L^{2}\theta^{\astrosun (2L-2)}=L^{2}\bb^{\astrosun (2-\frac{2}{L})}=\left[\nabla^{2}F(\bb)\right]^{-1}. \nonumber\end{align}Then, $\bb(t)$ follows the dynamic:
\begin{align}\Dot{\bb}(t)=-\left[\nabla^{2}F(\bb(t))\right]^{-1}\nabla\mathcal{L}(\bb(t)).\nonumber\end{align}
The previous flow coincides with the one derived in \cite{chou2023more,woodworth2020kernel}.  
\end{example}
\begin{example}
\label{Ex: UV} 
Consider the two layers vector-vector parameterization; i.e., $\bb=q(\theta)$ with $\theta=(\theu,\thed)\in \RR^k \times \RR^k$ and $q(\theta)=\theu\astrosun\thed$, which is indeed an overparameterization because there are more learnable parameters than variables to predict. Then, the gradient flow on the reparameterization is given by 
\begin{align}
     \Dot{\theu}(t)&=-\nabla_{\theu}\mathcal{L}(\theu(t) \astrosun \thed(t))=-\nabla\mathcal{L}(\bb(t))\astrosun \thed(t);\nonumber\\ 
     \Dot{\thed}(t)& =-\nabla_{\thed}\mathcal{L}(\theu(t) \astrosun \thed(t))=-\nabla\mathcal{L}(\bb(t))\astrosun \theu(t) \hspace{12mm}\text{ with }\theu^{\astrosun 2}(0)-\thed^{\astrosun 2}(0)=\overrightarrow{\gamma},
     \nonumber
\end{align}
 where $\overrightarrow{\gamma}\in\RR^p$. Let us compute the Jacobian of $q$
\begin{align}J_{q}(\theta)=\left[\Diag(\thed) \hspace{2mm}\Diag(\theu)\right]
\text{ and }\hspace{2mm}
    J_{q(\theta)}(\theta)J^{\top}_{q(\theta)}(\theta)=\Diag\left(\theu^{\astrosun2}+\theu^{\astrosun2}\right).\nonumber
\end{align} 
It is not straightforward to express the previous in terms of $\bb$. However, if we multiply the equation of $\Dot{\theu}(t)$ component-wise by $\theu(t)$ and we do the same for $\Dot{\thed}(t)$ by $\thed(t)$, we get:
\begin{align}
     \theu(t)\astrosun\Dot{\theu}(t)=-\nabla_{
     \bb}\mathcal{L}(\bb(t))\astrosun \theu(t)\astrosun \thed(t)=\thed(t)\astrosun\Dot{\thed}(t).\nonumber
\end{align}
From the previous equality and the initialization $\theu^{\astrosun 2}(0)-\thed^{\astrosun 2}(0)=\overrightarrow{\gamma}$, we deduce that, for every $t\geq 0$,
\begin{align}
 \theu^{\astrosun 2}(t)=\thed^{\astrosun 2}(t)+\overrightarrow{\gamma}.
 \label{theta1}
 \end{align}
 Multiplying the latter component-wise by $\thed^{\astrosun 2}(t)$ and recalling that $\bb^{\astrosun 2}(t)=\theu(t)\astrosun\thed(t)$, we obtain the following biquadratic equation:
 \begin{align}
      \thed^{\astrosun 4}(t)+\overrightarrow{\gamma}\astrosun \thed^{\astrosun 2}(t)-\bb^{\astrosun 2}(t)=0.\nonumber
 \end{align}
Since $\thed^{\astrosun 2}(t)\geq 0$, the previous equation has only one positive solution and from \eqref{theta1}, we have that:
\begin{align}
         \thed^{\astrosun 2}(t)=\frac{-\overrightarrow{\gamma}+\sqrt{\overrightarrow{\gamma}^{\astrosun 2}+4\bb^{\astrosun 2}(t)}}{2}\hspace{2mm} \text{ and } \hspace{2mm} \theu^{\astrosun 2}(t)=\frac{\overrightarrow{\gamma}+\sqrt{\overrightarrow{\gamma}^{\astrosun 2}+4\bb^{\astrosun 2}(t)}}{2}.\nonumber
\end{align}
Consider the following entropy function:
\begin{align}
    F(\bb):=\frac{1}{4}\sum\limits_{n=1}^{k}\left(2\bb_{n}\log\left(2\bb_{n}+\sqrt{\overrightarrow{\gamma}^{2}+4\bb_{n}^{2}}\right)-\sqrt{\overrightarrow{\gamma}^{2}+4\bb_{n}^{2}}\right). \nonumber
\end{align}
From the definition of the derivative of the product and the computation of the gradient it follows  that
\begin{align}
    \dot{\bb}(t)=\Dot{\theu}(t)\astrosun\thed(t)+\theu(t)\astrosun\Dot{\thed}(t)=-\left(\theu^{\astrosun 2}+\thed^{\astrosun 2}\right)\astrosun \nabla\mathcal{L}(\bb(t))=&-\sqrt{\overrightarrow{\gamma}^{2}+4\bb^{2}} \ \astrosun \ \nabla\mathcal{L}(\bb(t))\nonumber\\=&-\left[\nabla^{2}F(\bb(t))\right]^{-1}\nabla\mathcal{L}(\bb(t)).\nonumber
\end{align}
For more precise results about convergence and implicit bias in this example, see \cite{li2022implicit, varre2023spectral, kunin2024rich, azulay2021implicit}.
\end{example}

\section{Proofs}
\subsection{Proof of Theorem \ref{T: Existence}}\label{Append: Well-posedness}

In order to prove Theorem \ref{T: Existence}, we establish an auxiliary result. It states that both the Bregman distance and the loss function, along the trajectory of the flow, do not increase. 
\begin{lemma}\label{L: Loss-Decreasing} Suppose that there exists a solution $\bb(t)$ of \ref{eq: x(t)flow} for $t \in \left[0,T\right)$. Suppose that Assumptions \ref{A: D2} and \ref{A: D1} hold, $\bb(t) \in \RR^p \setminus\left\{b\right\}$ for every $t \in \left[0,T\right)$. Then, both the Bregman divergence $D_{F}(z,\bb(t))$ for every $z\in \mathcal{S}$ and the functional value $\mathcal{L}(\bb(t))$  are non-increasing functions for $t\in [0,T)$.
\end{lemma}\vspace{3mm}
\begin{proof}
The proof is the same in both cases. To prove that $t\mapsto D_{F}(z,\bb(t))$ is non-increasing, it is sufficient to note that, for every $t\in [0,T)$,
\begin{align}
\frac{d}{dt} D_{F}(z,\bb(t))=&\hspace{1mm}-\scal{\nabla F(\bb(t))}{\dot{\bb}(t)} -\scal{\nabla^{2} F(\bb(t))\dot{\bb}(t)}{z-\bb(t)}+\scal{\nabla F(\bb(t))}{\dot{\bb}(t)}\nonumber\\
=&\hspace{1mm}-\scal{\nabla^{2} F(\bb(t))\dot{\bb}(t)}{z-\bb(t)}\nonumber\\
=&\hspace{1mm}G(\bb(t))\scal{\nabla\mathcal{L}(\bb(t))}{z-\bb(t)}\label{eq: derivative DF}\\ 
\leq& \ G(\bb(t))(\mathcal{L}(z)-\mathcal{L}(\bb(t)))\nonumber\\ \leq&\hspace{1mm} 0,\nonumber
\end{align}
where we used that $\bb$ is a solution of \eqref{eq: x(t)flow}, the gradient inequality for the convex function $\mathcal{L}$ and the fact that, for every $t\in [0,T)$,
 $\mathcal{L}(\bb(t))\geq \mathcal{L}(z)$ and $G(\bb(t))\geq 0$.
On the other hand, for $t\in [0,T)$, $\nabla^2 F(\bb(t))$ is invertible (and the inverse is positive definite). Then,
\begin{align}
    \frac{d}{dt} \mathcal{L}(\bb(t))=&\hspace{1mm}\scal{\nabla\mathcal{L}(\bb(t))}{\dot{\bb}(t)}\nonumber  \\=&\hspace{1mm}- G(\bb(t))\scal{\nabla\mathcal{L}(\bb(t))}{ \left[\nabla^2 F(\bb(t))\right]^{-1}\nabla\mathcal{L}(\bb(t))}\nonumber\\ \leq &\hspace{1mm}0,\nonumber
\end{align}
which completes the proof.\hfill\end{proof}

We state another preliminary result, which will be crucial for proving Theorem~\ref{T: Existence} under Assumption~\ref{Assumption}. It consists in a lower bound for the norm of $\bb-b$ along the trajectory.

\begin{lemma}\label{L: Bounded away x(t)}
Suppose that assumptions \ref{A: D2}, \ref{A: D1}, \ref{A: D3} hold. Suppose that there exists a solution $\bb(t)$ of \ref{eq: x(t)flow} for $t \in \left[0,T\right)$. Set $$r:=\frac{\mathcal{L}(b)-\mathcal{L}(\bb_0)}{2 \|\nabla \mathcal{L}(b)\|}.$$ 
Then, for every $t\in[0,T)$,
\begin{equation}\label{Eq: Bounded BelowG}
    \|\bb(t)-b\| > r \quad \text{and} \quad G(\bb(t))\geq C(r).
\end{equation}
\end{lemma}

\begin{proof}
Let $\bb \in \R^p$ be such that $\|\bb-b\|\leq r$. Then, the gradient inequality for the convex function $\mathcal{L}$ at the point $0$, the Cauchy-Schwarz inequality, and Assumptions~\ref{A: D3} imply
\begin{equation}\label{eq: Bounded away L1}
\mathcal{L}(\bb) \geq \mathcal{L}(b)+\scal{\nabla \mathcal{L}(b)}{\bb-b}  \geq \ \mathcal{L}(b) - r \norm{\nabla \mathcal{L}(b)} \geq \frac{\mathcal{L}(b)+\mathcal{L}(\bb_0)}{2} > \mathcal{L}(\bb_0). 
\end{equation}
Next, by contradiction, suppose that there exists a $\hat{t}\in [0,T)$ such that $\|\bb\big(\hat{t}\big)-b\|\leq r$. 
Then, by \eqref{eq: Bounded away L1}, we have that $\mathcal{L}(\bb_0) ~<~\mathcal{L}(\bb(\hat{t}))$. This is a contradiction since by Lemma~\ref{L: Loss-Decreasing} $t\mapsto\mathcal{L}(\bb(t))$ is non-increasing on $[0,T)$. Then $\|\bb\left(t\right)-b\|>r$ for every $t\in[0,T)$. The second part of the statement in \eqref{Eq: Bounded BelowG} is a consequence of the condition on $G$ assumed in \ref{A: D3}. 
\hfill\end{proof}
\ \\

\begin{proof}[Theorem \ref{T: Existence}]\\
\ \\
\textbf{Existence.}
Denote by $T_{M}$ the maximal time for the existence of a solution of \eqref{eq: x(t)flow} in $\RR^{p}\setminus\left\{b\right\}$; namely,
    \begin{align}
        T_{M}:=\sup\left\{T>0\mid \exists \text{ a } \, C^1 \text{ function } \bb\colon [0,T)\to \RR^{p}\setminus\left\{b\right\}  \text{ solution of \eqref{eq: x(t)flow} }\right\}.\nonumber
    \end{align} 
First we show that $T_M>0$. Since $\nabla^2 F$ is invertible on $\RR^{p}\setminus\left\{b\right\}$, the dynamical system in \eqref{eq: x(t)flow} is equivalent to the following:
\begin{align}
    \label{eq: x(t)flowI}
\begin{cases}
\dot{\bb}(t)=-G(\bb(t))
\left [\nabla^{2}F(\bb(t))\right]^{-1}\nabla\mathcal{L}(\bb(t)), \quad\quad  0<t<T_M\\
\bb(0)=\bb_0\in \RR^{p}\setminus\left\{b\right\}.
\end{cases}
\end{align}
 Since $G(\bb)$, $\left[\nabla^{2}F(\bb)\right]^{-1}$, and $\nabla \LLL(\bb)$ are locally Lipschitz on $\RR^{p}\setminus\left\{b\right\}$, they are continuous and bounded on bounded subsets and their product is also locally Lipschitz. Since $\bb_{0}\in \RR^{p}\setminus\left\{b\right\}$, it follows from \cite[Theorem 1.18]{barreira2012ordinary} that there exists a unique solution of the initial value problem \eqref{eq: x(t)flowI}, and so to \eqref{eq: x(t)flow}, in some time-interval. Then, $T_M>0$. Now we proceed to prove the claim that $T_M=+\infty$; namely, that the solution of the dynamical system~\eqref{eq: x(t)flow} is globally defined. Lemma~\ref{L: Bounded away x(t)} yields $\norm{\bb(t)-b}>r$ for every $t\in\left[0,T_M\right)$. Lemma \ref{L: Loss-Decreasing} implies that $D_{F}(z,\bb(t))$ is non-increasing for every $z\in \mathcal{S}$ and, by Assumption \ref{A: D4}, $D_{F}(\bar{z},\cdot)$ is coercive for some $\bar{z}\in\mathcal{S}$. Then, $\bb([0,T_{M}))$ is bounded. Therefore, $\bb([0,T_{M}))$ is contained in a compact subset of $\RR^{p}\setminus\left\{b\right\}$. We then derive from \cite[Proposition 1.57]{barreira2012ordinary} that $T_M=+\infty$. \\
\textbf{Uniqueness.} Since the right-hand side of \eqref{eq: x(t)flowI} is locally Lipschitz
at every point, local uniqueness holds for every initial state, and therefore globally.\hfill 
\end{proof}
\vspace{4mm}
\subsection{Proof of Theorem \ref{T: Main result}}\label{Append: Main result}

Let $z\in \mathcal{S}$. By the Fundamental Theorem of Calculus, inequality \eqref{eq: derivative DF}, we have that, for all $t\geq 0$,
\begin{align}\label{breg_div}
   D_{F}(z,\bb(t))-D_{F}(z,\bb_0)=&\int_{0}^{t}\frac{d}{ds}D_{F}(z,\bb(s)) \ ds\\
   \leq& -\int_{0}^{t}G(\bb(s))(\LLL(\bb(s))-\LLL^{*})ds \nonumber \\
   \leq& -C(r)\int_{0}^{t}(\LLL(\bb(s))-\LLL^{*})ds
   \leq0\nonumber,
   \end{align} 
where the function $C$ is the one in Assumption \ref{A: D2} and $r$ is defined as in Lemma \ref{L: Bounded away x(t)}.
Since $D_{F}(z,\bb(t))$ is non-negative, from the above inequality we derive that, for every $t\geq 0$,
\begin{equation*}
   \int_{0}^{t}(\LLL(\bb(s))-\LLL^{*})ds\leq \frac{D_{F}(z,\bb_0)}{C(r)}.
  \end{equation*}
In particular, $t\mapsto \left[\LLL(\bb(t))-\LLL^{*}\right]$ belongs to $L^1\left(\left[0,+\infty\right)\right)$. Moreover, since $t\mapsto \LLL(\bb(t))$ is non-increasing by Lemma \ref{L: Loss-Decreasing}, we have that $\left(\LLL(\bb(t))-\LLL^{*}\right)\cdot t \leq \int_{0}^t (\LLL(\bb(s))-\LLL^{*})ds\leq D_{F}(z,\bb_0)/C(r)$ for every $t\geq 0$. From the latter, we obtain the rate stated in the theorem.\\
Next, we prove the convergence of the trajectory $(\bb(t))_{t>0}$. Let $\bar{z}$ be the point in Assumption \ref{A: D4}. By Lemma \ref{L: Loss-Decreasing}, $D_{F}(\bar{z},\bb(t))$ is (non-negative and) non-increasing, and therefore there exists $\lim\limits_{t\rightarrow +\infty}D_{F}(\bar{z},\bb(t))\in\RR$. Since $D_{F}(\bar{z},\cdot)$ is coercive, we have that $\bb(t)$ is bounded. Then, there exist a sequence $\left(t_{k}\right)_{k\geq0}$ and a point $\bb_*$ such that, for $k\to +\infty$, $t_k\to +\infty$ and $\bb(t_k)\to\bb_*$. Since $\LLL$ is continuous and $\LLL(\bb(t))-\LLL^{*} \leq D_{F}(z,\bb_0)/(ct)$, we deduce that 
\begin{equation*}\LLL(\bb_*)=\lim\limits_{k\to + \infty} \LLL(\bb(t_k))= \lim\limits_{t\to +\infty} \LLL(\bb(t))=\LLL^*;\end{equation*} 
i.e., that $\bb_*\in\mathcal{S}$. We are going to show that indeed $\lim\limits_{t\rightarrow \infty}\bb(t) =  \bb_*$. First notice that $\beta_*$ is in $\mathcal{S}$ and so $\beta_*\neq b$. Suppose that there exists another sequence $\left(t'_{n}\right)_{n=1}^{\infty}$ such that $t'_n\to +\infty$ and $\lim\limits_{n\to +\infty}\bb(t'_n) = \bar{\bb}$. With the same reasoning as above, we know that $\bar{\bb}$ belongs to $\mathcal{S}$ and so $\bar{\beta} \neq b$. Continuity of $D_F(\bb_*, \cdot)$ in $\bar{\beta}$ and in $\bb_*$ and the existence of $\lim\limits_{t\to +\infty}D_{F}(\bb_*,\bb(t))$ yields
\begin{align}
 D_F(\bb_*, \bar{\bb})=&\lim\limits_{n\to +\infty}D_{F}(\bb_*,\bb(t'_{n}))=\lim\limits_{t\to +\infty}D_{F}(\bb_*,\bb(t))\nonumber\\=&\lim\limits_{k\to +\infty}D_{F}(\bb_*,\bb(t_k))=D_{F}(\bb_*,\bb_*)=0.\nonumber
\end{align}
Since $F$ is strictly convex, we get $\bar{\bb}=\bb_*$ and therefore $\lim\limits_{t\rightarrow +\infty}\bb(t)=\bb_*$.
To conclude, note that trivially
\begin{align}    \bb_*\in&\argmin\limits_{z\in\mathcal{S}} D_{F}(z,\bb_*). \nonumber
\end{align}
By \eqref{breg_div} and \eqref{eq: derivative DF}, 
\begin{align}
    D_{F}(z,\bb_*)&=\lim\limits_{t\to +\infty} D_{F}(z,\bb(t))\nonumber\\
    &= \lim\limits_{t\to +\infty} D_F(z,\bb_0)+\int_{0}^{t}G(\bb(s))\scal{\nabla\mathcal{L}(\bb(s))}{z-\bb(s)}ds\nonumber \\
    & = F(z)-F(\bb_0)-\scal{\nabla F(\bb_{0})}{z-\bb_0}+\int_{0}^{+\infty}G(\bb(s))\scal{\nabla\mathcal{L}(\bb(s))}{z-\bb(s)}ds.\nonumber
\end{align}
Discarding the terms that are constant with respect to $z$, we obtain the claim.\hfill

\subsection{Proof of Corollary~\ref{L: Dsol-Uv}}\label{P: Dsol-Uv}
\begin{lemma}\label{lemma_alga}
    Let $X: \ \R^p \to \R^d$ linear, $y\in\R^d$ and $\beta \in \R^p$ such that $X\beta = y$. Let
    \begin{align}
    \{ \bb_{*} \} = \argmin_{Xz=y}\ \|z\|^2.
    \label{P: minimal norm UvI}
\end{align}
Let us define $P_{\perp}$ as the projection onto $\ker(X)^{\perp}$.
Then, 
$$P_{\perp}(\bb)=\bb_*.$$
    
\end{lemma}

\begin{proof}
The optimality conditions for problem \eqref{P: minimal norm UvI} imply that there exists $\eta_*\in \mathbb{R}^d$ such that $\beta_*=X^{\top}\eta_*$ and so that $\bb_{*}\in \ker(X)^{\perp}$.
Moreover, since $X\bb=y=X\bb_{*}$, we obtain that $\bb-\bb_{*}\in \ker X$. Then, since $\bb^*\in\ker(X)^{\perp}$,
we derive from
\begin{align}
   \left(\forall p\in \ker(X)^{\perp}\right) \scal{p-\bb_{*}}{\bb-\bb_{*}}=0\nonumber
\end{align}
that $P_{\perp}(\bb)=\bb_{*}$. \hfill
\end{proof}
\ \\
\ \\
\begin{proof}[Corollary ~\ref{L: Dsol-Uv}]
As shown in Corollary~\ref{cor:7.2} all the assumptions of Theorem \ref{T: Existence} and Theorem \ref{T: Main result} are satisfied. Then equation \ref{eq: Uvanot0 Flow H}   has a unique solution defined on $[0,+\infty)$ satisfying $\beta(t)>0$ and $\lim\limits_{t\rightarrow +\infty}\bb(t)=\bb_{\infty}$. Let us define $P_{\perp}=\Id-P$, the projection onto $\ker(X)^{\perp}$. Then we can write $\bb_{\infty}$ as
 \begin{equation}
\label{proj_split}
\bb_{\infty}=P(\bb_{\infty})+P_{\perp}(\bb_{\infty}).
 \end{equation}
 First, we compute $P(\bb_{\infty})$. Since $F(\beta)=(2/3)\|\beta\|^{3/2}$, we have $\nabla F(\beta)=\beta/\|\beta\|^{1/2}$, and Corollary~\ref{Remark: Saving energy}
implies that 
\begin{align}
    0 = \frac{d}{dt} P\left(\nabla F(\beta(t))\right) = \frac{d}{dt} \left(\frac{P(\beta(t))}{\|\beta(t)\|^{1/2}}\right).
\end{align}
Thus the term $\frac{P(\bb(t))}{{\|\bb(t)\|^{1/2}}}$ remains constant for every $t>0$, and equal to $\frac{P(\bb_{0})}{{\|\bb_{0}\|^{1/2}}}$. Consequently
\begin{align}
P\left(\bb(t)\right)=P(\bb_{0})\left({\frac{\|\bb(t)\|}{\|\bb_{0}\|}}\right)^{1/2},\nonumber
\end{align}
and from the continuity of the projection and of the norm, taking the limit for $t\to+\infty$ in the previous equation, we get
\begin{equation*}    P(\bb_{\infty})=P(\bb_{0})\left({\frac{\|\bb_{\infty}\|}{\|\bb_{0}\|}}\right)^{1/2}.
\end{equation*}
Since $\beta_\infty\in\mathcal{S}$, thanks to Lemma \ref{lemma_alga},  we get
\begin{equation*}\bb_{\infty}=P(\bb_{\infty})+P_{\perp}(\bb_{\infty})=P(\bb_{0})\sqrt{\frac{\|\bb_{\infty}\|}{\|\bb_{0}\|}}+\bb_{*}.\end{equation*}
Finally, we obtain the value of $\|\bb_{\infty}\|$ and $\alpha$ by solving the following second order equation:
\begin{align}\|\bb_{\infty}\|^{2}=\|P_{\perp}(\bb_{\infty})\|^{2}+\|P(\bb_{\infty})\|^{2}=\|\bb_{*}\|^{2}+\|P(\bb_{0})\|^{2}\frac{\|\bb_{\infty}\|}{\|\bb_{0}\|}.\nonumber\end{align}\hfill
\end{proof}
\subsection{Proof of Lemma \ref{L: Hessian-Norm}}\label{Proof: Hessian-Norm}
\begin{proof} The function $F$ is strictly-convex since it is the composition of the increasing strictly-convex function $f$ with the convex function $\bb\mapsto\|\bb\|$. The gradient and the Hessian of $F(\bb) = f(\|\bb\|)$, for $\bb\neq 0$, are
\begin{align}
   \nabla F(\bb)=&f'(\|\bb\|)\frac{\bb}{\|\bb\|} \hspace{2mm}\text{ and }\hspace{2mm}\nonumber\\ 
   \nabla^{2}F(\bb)=&f''(\|\bb\|)\frac{\bb \bb^{\top}}{\|\bb\|^2}+\frac{f'(\|\bb\|)}{\|\bb\|}\left(\Id-\frac{\bb\bb^{\top}}{\|\bb\|^{2}}\right),\label{Hessian of the norm}
\end{align}
and clearly $F$ is twice differentiable on $\RR^p\setminus\{0\}$.
Let $M=\frac{f'(\|\bb\|)}{\|\bb\|}\Id$, $u=\left(f''(\|\bb\|)-\frac{f'(\|\bb\|)}{\|\bb\|}\right)\frac{\bb}{\|\bb\|^{2}}$ and $v=\bb$. Notice that, since $f'>0$ by assumption, $M$ is invertible. Since by the assumptions $v^{\top}M^{-1}u+1=\frac{\|\bb\|f''(\|\bb\|)}{f'(\|\bb\|)}\neq 0$ for every $\bb \in \RR^p\setminus\{0\}$, the Sherman-Morrison formula \eqref{Sherman} yields
\begin{align}
\label{eq:invH}
   \left[\nabla^{2}F(\bb)\right]^{-1}=\frac{1}{f''(\|\bb\|)}\frac{\bb \bb^{\top}}{\|\bb\|^2}+\frac{\|\bb\|}{f'(\|\bb\|)}\left(\Id-\frac{\bb\bb^{\top}}{\|\bb\|^{2}}\right).
\end{align}
The first statement of the theorem follows by inverting the Hessian in equation \eqref{eq: x(t)flow} and plugging the explicit form of $\left[\nabla^{2}F(\bb)\right]^{-1}$. Now we want to show that, under the additional hypothesis of the theorem, Assumption~\ref{A: D2} holds. 
The function $\left[\nabla^{2}F(\bb)\right]^{-1}$ is locally Lipschitz in $\RR^p\setminus\{0\}$ since $1/f''$ and $1/f'$ are locally Lipschitz in $(0,+\infty)$. Local lipschitzianity of $g$ implies that $G$ is locally Lipschitz. % Assumption~\ref{A: D3i} is also trivially satisfied, therefore Lemma~\ref{L: Bounded away x(t)} implies that the solution of the dynamical system is bounded away from zero; namely, that $\|\bb(t)\| \geq r > 0$ for every $t \geq  0$. This implies that all the quantities in~\eqref{Hessian of the norm} and in~\eqref{eq:invH} are well-defined along the trajectory. 
Last part of Assumption~\ref{A: D2} is satisfied trivially by the hypothesis on $g$. 
\hfill
\end{proof}
\subsection{Proof of Theorem \ref{T: dynamics of x(t) Fwn}}\label{Proof: dynamics of x(t) Fwn}\vspace{2mm}
\begin{proof}
 We first show that $t\in(0,+\infty)\mapsto \|\w(t)\|$ is constant. Since $w$ is a solution of the dynamical system in \eqref{eq: FWnormalvector-Flow-u}, we obtain
\begin{align}
     \frac{d}{dt}\|\w(t)\|^{2}=2\scal{\w(t)}{\Dot{\w}(t)}=0,\nonumber
\end{align}
therefore the norm of $\w(t)$ is constant. Moreover, since $\|\w(0)\|=\gamma_0$ by assumption, for every $t>0$, $\|\w(t)\|=\gamma_0$. Then, along the trajectory, 
$\gamma_0\bb(t)=h(\alpha(t))w(t)$. Next, we show that $h(\alpha(t))>0$ for every $t>0$. We first derive an equation for $\bb(t)$ using \eqref{eq: FWnormalvector-Flow-u}-\eqref{eq: FWnormalvector-Flow-alpha} in order to show that $\mathcal{L}(\beta(t))$ is non-increasing. The derivative of the product yields
\begin{align}
\gamma_0\dot{\bb}(t)=&h(\alpha(t))\dot{\w}(t)+h'(\alpha(t))\dot{\alpha}(t)\w(t)\nonumber\\=&-\frac{\left(h(\alpha(t)\right)^{2}}{\|\w(t)\|}\left(\Id-\frac{\w(t)\w^{\top}(t)}{\|\w(t)\|^{2}}\right)\nabla\mathcal{L}(\bb(t))- \frac{\left(h'(\alpha(t))\right)^{2}}{\|\w(t)\|}\w(t)\w^{\top}(t)\nabla\mathcal{L}(\bb(t)).\label{eq: FW-eq-Normal-wn}
\end{align}
It follows from \eqref{eq: FW-eq-Normal-wn} and Cauchy-Schwarz inequality that 
\begin{align}
    \gamma_0\frac{d}{dt} \mathcal{L}(\bb(t))=&\hspace{1mm}\scal{\nabla\mathcal{L}(\bb(t))}{\gamma_0\dot{\bb}(t)}\nonumber \\=&-\frac{\left(h(\alpha(t)\right)^{2}}{\|\w(t)\|}\|\nabla\mathcal{L}(\bb(t))\|^{2}+\frac{\left(h(\alpha(t)\right)^{2}}{\|\w(t)\|}\scal{\nabla\mathcal{L}(\bb(t))}{\frac{\w(t)}{\|\w(t)\|}}^{2}\nonumber \\&- \frac{\left(h'(\alpha(t))\right)^{2}}{\|\w(t)\|}\scal{\nabla\mathcal{L}(\bb(t))}{\w(t)}^{2}\nonumber \\ \leq&- \frac{\left(h'(\alpha(t))\right)^{2}}{\|\w(t)\|}\scal{\nabla\mathcal{L}(\bb(t))}{\w(t)}^{2}\nonumber\\ \leq &\hspace{1mm}0.\nonumber
\end{align}
Therefore $t\in(0,+\infty)\mapsto \mathcal{L}(\bb(t))$ is decreasing. Next, we want to show that $\|\beta(t)\|\geq \tilde{r}$ for every $t>0$. Indeed, recalling that $|h(\alpha(t))|=\|\beta(t)\|$, $h\circ \alpha$ is continuous and $h(\alpha_0)>0$, we derive $h(\alpha(t))>\tilde{r}$ for every $t>0$. To this aim, let $\beta$ with $\|\bb\|\leq \tilde{r}$. Then, convexity of $\mathcal{L}$ and the Cauchy-Schwarz inequality yield
\begin{align}
\mathcal{L}(\bb)  &  \geq \mathcal{L}(0)+\scal{\nabla \mathcal{L}(0)}{\bb-0}  \geq \ \mathcal{L}(0) - \tilde{r} \norm{\nabla \mathcal{L}(0)} \ = \ \frac{\mathcal{L}(0)+\mathcal{L}(\bb_0)}{2} \ > \ \mathcal{L}(\bb_0). \label{eq: Bounded away L2}
\end{align}
This implies on the one hand that $\|\bb_0\|=\|h(\alpha_0)w_0/\|w_0\|\|=h(\alpha_0)>\tilde{r}$, otherwise we would derive a contradiction from \eqref{eq: Bounded away L2} choosing $\bb=\bb_0$. On the other hand, from \eqref{eq: Bounded away L2} we derive that $\|\bb(t)\|>\tilde{r}$ for every $t>0$ since $\mathcal{L}(\bb(t))$ is decreasing.

Then, for every $t\geq 0$, $|h\left(\alpha(t)\right)|=h(\alpha(t))>\tilde{r}$.
Replacing $\w(t)=\frac{\gamma_0\bb(t)}{\|\bb(t)\|}$, $\|\w(t)\|=\gamma_0$, $\|\bb(t)\|=h(\alpha(t))$, and
\begin{equation*}(h'(\alpha))^2=\frac{ f'(h(\alpha))h(\alpha)}{\gamma_0^2 f''(h(\alpha))},\end{equation*}
in \eqref{eq: FW-eq-Normal-wn}, we deduce that
\begin{align}\Dot{\bb}(t)=&-\frac{\|\bb(t)\|^{2}}{\gamma_0^2}\left(\Id-\frac{\bb(t)\bb^{\top}(t)}{\|\bb(t)\|^{2}}\right)\nabla\mathcal{L}(\bb(t))\nonumber\\&- \frac{f'(\|\bb(t)\|)\|\bb(t)\|}{\gamma_0^2 f''(\|\bb(t)\|)}\frac{\bb(t)\bb^{\top}(t)}{\|\bb(t)\|^2}\nabla\mathcal{L}(\bb(t)).\nonumber\end{align}
Then, Lemma \ref{L: Hessian-Norm} with $F(\bb)=f(\|\bb\|)$ and $G(\bb)=g(\|\bb\|)=\|\bb\| f'(\|\bb\|)/\gamma_0^2$ shows that $\beta$ satisfies \eqref{eq: x(t)flow}.
It remains to show that assumption \ref{A: D4} holds. Indeed, we prove that $D_F(z,\cdot)$ is coercive for every $z \in \R^d$. Note that $f'$ is increasing in $s > 0$ and so $f$ is convex. Then, for every $s > 0$ and $\beta \in \R^d \setminus \left\{0\right\}$,
\begin{align*}
  f(s)+f'(s)(\|\bb\|-s)\leq f(\|\bb\|).  
\end{align*}
 Dividing the previous inequality by $\|\bb\|$ yields
\begin{align*}
    f'(s)= \lim\limits_{\|\bb\|\rightarrow\infty}\frac{f(s)+f'(s)(\|\bb\|-s)}{\|\bb\|}\leq \liminf\limits_{\|\bb\|\rightarrow\infty}\frac{f(\|\bb\|)}{\|\bb\|}= \liminf\limits_{\|\bb\|\rightarrow\infty}\frac{F(\bb)}{\|\bb\|}.
\end{align*} 
Since $s>0$ is arbitrary and $f'(s)\to+\infty$ as $s\to +\infty$, we deduce that $F$ is supercoercive. It follows from \cite[Lemma 7.3 (viii)]{bauschke2001essential} that the $D_{F}(z,\cdot)$ is coercive for every $z\in \R^d$, which implies assumption \ref{A: D4}.
\hfill
\end{proof}
\subsection{Proof of Lemma \ref{L: CWNS}}\label{P: CWNS}
\begin{proof}
Let $P_{\perp}:=\Id-P$ be the projection onto $\ker(X)^{\perp}$. Then
 \begin{equation*}
     \bb_{\infty}=P(\bb_{\infty})+P_{\perp}(\bb_{\infty}).
 \end{equation*}
First, we compute $P(\bb_{\infty})$. It follows from Theorem~\ref{T: dynamics of x(t) Fwn} that $\beta(t)$ solves \eqref{eq: x(t)flow} with $F(\beta)=f(\|\beta\|)$.  Corollary~\ref{Remark: Saving energy} then implies that 

 \begin{align}
 \label{eq:betadot}
0= \frac{d}{dt}\left[P(\nabla F(\beta(t)))\right]=\frac{d}{dt}  \left[\frac{f'(\|\beta(t)\|)P(\beta(t))}{\|\beta(t)\|}\right].  
\end{align}
Thus the term $\frac{f'(\|\bb(t)\|)P(\bb(t))}{\|\bb(t)\|}$ remains constant for every $t>0$ and it is equal to $\frac{f'(\|\bb_{0}\|)P(\bb_{0})}{\|\bb_{0}\|}$. Consequently, we obtain the following expression:
\begin{align}
P(\bb(t))=P\left(\bb_{0}\right)\frac{f'(\|\bb_{0}\|) \ \|\bb(t)\|} {f'(\|\bb(t)\|) \ \|\bb_{0}\|}.\nonumber
\end{align}
Recall that, from Theorem \ref{T: Main result}, $\lim\limits_{t\rightarrow +\infty}\bb(t)=\bb_{\infty}$. From the continuity of the projection,  the continuity of the norm and of $f'$, taking the limit for $t\to+\infty$ in the previous equation, we get
\begin{equation*}
P(\bb_{\infty})=P\left(\bb_{0}\right)\frac{f'(\|\bb_{0}\|) \ \|\bb_{\infty}\|} {f'(\|\bb_{\infty}\|) \ \|\bb_{0}\|}.
\end{equation*}
Since $\bb_{\infty}\in\mathcal{S}$,
from Lemma \ref{lemma_alga} we know that $P_{\perp}(\bb_{\infty})=\bb_{*}$. Then,
\begin{equation*}\bb_{\infty}=P(\bb_{\infty})+P_{\perp}(\bb_{\infty})=P\left(\bb_{0}\right)\frac{f'(\|\bb_{0}\|) \ \|\bb_{\infty}\|} {f'(\|\bb_{\infty}\|) \ \|\bb_{0}\|}+\bb_{*}.\end{equation*}
\hfill
\end{proof}

\subsection{Proof of Theorem \ref{T: dynamics of x(t) aUV}}\label{P: dynamics of x(t) aUV}
In this section, we provide the proof of Theorem \ref{T: dynamics of x(t) aUV}.\\ \\
\begin{proof}
If we multiply \eqref{eq: WMatrixVector-Flow-v} by $\w^{\top}(t)$ on the right and \eqref{eq: WMatrixVector-Flow-U} by $\W^{\top}(t)$ on the left, we obtain that
\begin{align}
   \W^{\top}(t)\Dot{\W}(t)=\Dot{\w}(t)\w^{\top}(t).
    \label{eq: eqaulity matrixW}
\end{align}
Adding the transpose of \eqref{eq: eqaulity matrixW} yields
\begin{align}
    \W^{\top}(t)\Dot{\W}(t)+ \dot{\W}^{\top}(t)\W(t)=\Dot{\w}(t)\w^{\top}(t)+\w(t)\dot{w}^{\top}(t).
    \nonumber
\end{align}
Since $\W^{\top}(0)\W(0)=\w^{\top}(0)\w^{\top}(0)$, we obtain
\begin{align}
    \W^{\top}(t)\W(t)=\w(t)\w^{\top}(t).\label{eq: eqaulity matrixW1W2}
\end{align}
Multiplying \eqref{eq: eqaulity matrixW1W2} by $\w^{\top}(t)$ on the left and by $\w(t)$ on the right, we have that $\|\w(t)\|^{4}=\|\W(t)\w(t)\|^2$,
for every $t$. Thus, $\|\w(t)\|^2=\|\W(t)\w(t)\|$.
Analogously, multiplying by $\W(t)$ the equation \eqref{eq: eqaulity matrixW1W2} on the left and by $\W^{\top}(t)$ on the right, we deduce that
\begin{align}\label{eq:sqroot}
(\W(t)\W^{\top}(t))^{2}=\frac{\bb(t)\bb^{\top}(t)}{\|\bb(t)\|^{2}}\|\W(t)\w(t)\|^2.\end{align}
In addition, since
$W(t)W^T(t)$ is positive semi-definite and symmetric it is the unique square root of the right-hand side of \eqref{eq:sqroot} and therefore 
\[
 \W(t)\W^{\top}(t)=\frac{\bb(t)\bb^{\top}(t)}{\|\bb(t)\|^{2}}\|\W(t)\w(t)\|.
\]
Moreover, note that
\begin{align}
    \Dot{\W}(t)\w(t)+\W(t)\Dot{\w}(t)=&-\frac{h(\alpha(t))}{\|\W(t)\w(t)\|}\left(\|\w(t)\|^2\Id  +\W(t)\W^{\top}(t)\right)\left(\Id-\frac{\bb(t)\bb^{\top}(t)}{\|\bb(t)\|^2}\right)\nabla\LLL(\bb(t))\nonumber\\=& -h(\alpha(t))\left(  \Id+\frac{\bb(t)\bb^{\top}(t)}{\|\bb(t)\|^2}\right)\left(\Id-\frac{\bb(t)\bb^{\top}(t)}{\|\bb(t)\|^2}\right)\nabla\LLL(\bb(t))\nonumber\\=& -h(\alpha(t))\left(\Id-\frac{\bb(t)\bb^{\top}(t)}{\|\bb(t)\|^2}\right)\nabla\LLL(\bb(t))\nonumber
\end{align}
and \begin{align}
\partial_{t}\frac{\|\W(t)\w(t)\|^2}{2}&= \scal{\W(t)\w(t)}{\Dot{\W}(t)\w(t)+\W(t)\Dot{\w}(t)}_{F}\nonumber\\=&\scal{\frac{\|\W(t)\w(t)\|\bb(t)}{\|\bb(t)\|}}{\Dot{\W}(t)\w(t)+\W(t)\Dot{\w}(t)}_{F}\nonumber\\=& -\|\W(t)\w(t)\|{\bb^{\top}(t)}\left(\Id-\frac{\bb(t)\bb^{\top}(t)}{\|\bb(t)\|^2}\right)\nabla\LLL(\bb(t))\nonumber=0,
\end{align}
which implies that $\|\W(t)\w(t)\|=1$ for every $t$. Then,
\begin{align}
    \dot{\bb}(t)&=\dot{\alpha}(t)h'(\alpha(t))\W(t)\w(t)+h(\alpha(t))[\Dot{\W}(t)\w(t)+\W(t)\Dot{\w}(t)]\nonumber\\&=-(h'(\alpha(t)))^{2}\frac{\bb(t)\bb^{\top}(t)}{\|\bb(t)\|^{2}}\nabla\LLL(\bb(t))-\|\bb(t)\|^{2}\left(   \Id- \frac{\bb(t)\bb^{\top}(t)}{\|\bb(t)\|^2}\right)\nabla\LLL(\bb(t)).\nonumber
\end{align}
The statement then follows from~\eqref{A5: f-g-h assumption radial1} and Lemma~\ref{L: Hessian-Norm}.
\hfill
\end{proof}

\section{The Legendre case}\label{sec:legendre}
In this section we propose the main results under a different assumption on the function $F$. In particular, $F\colon \R^p\to\R\cup\{+\infty\}$ is an extended valued function. In this case the proof techniques for the results are known (see \cite{alvarez2004hessian} and \cite{bauschke2017descent}), and we rephrase them in our setting for the reader's convenience. 
 \begin{definition}\cite[Definition 3.1]{alvarez2004hessian}\label{Def: Legendre} Let $F\colon \RR^p \to \RR\cup\{+\infty\}$ be proper convex and l.s.c. such that $\operatorname{int}\left(\dom \left(F\right)\right)\neq \emptyset$, where we recall that $\mathrm{dom}(F)=\left\{\beta \in \R^p\colon F(\beta)<+\infty \right\}$. We say that $F$ is essentially smooth if it satisfies the following two conditions 
       \begin{enumerate}
           \item $F$ is differentiable on $\operatorname{int}\left(\dom \left(F\right)\right)$
           \item For every sequence $\{\bb_{k}\}_{k=1}^{+\infty}\subset \operatorname{int}\left(\dom \left(F\right)\right)$ converging to the boundary of $\dom \left(F\right)$, we get that $\lim\limits_{k\rightarrow +\infty}\|\nabla F(\bb_{k})\|\rightarrow+\infty$.
\end{enumerate}
In addition, $F$ is a Legendre type (or simply Legendre) function if it is essentially smooth and strictly convex on $\operatorname{int}\left(\dom \left(F\right)\right)$.
\end{definition}

\begin{definition}\cite{bregman1966relaxation}\label{Def:BD}
Let $F\colon\RR^{p}\mapsto \RR\cup\{+\infty\}$ be strictly
convex, proper, and G\^ateaux differentiable on $\operatorname{int}(\dom F )\neq \emptyset$. The Bregman divergence $D_{F}$ associated with $F$ is defined as:
\begin{align}
    D_{F}:\RR^{p}\times\RR^{p}\mapsto &\RR\cup\{+\infty\}\nonumber\\(y,\bb)\mapsto &\left\{
    \begin{array}{lr}
          F(y)-F(\bb)-\scal{\nabla F(\bb)}{y-\bb} &  \text{if }\bb\in\operatorname{int}(\dom(F)),\\
       +\infty  & \text{otherwise.}
    \end{array}
  \right.
\end{align}
\end{definition}

Define
\[
\mathcal{S}:=\argmin_{\bb \in{\mathrm{dom}(F)}}\mathcal{L}(\bb) \quad \quad \text{and} \quad \quad \mathcal{L}^*:=\inf_{\bb \in \mathrm{dom}(F)}\mathcal{L}(\bb).
\]
Note that, since $\mathrm{dom}(F)$ is not necessarily the full space $\R^p$, $\mathcal{S}$ and $\mathcal{L}^*$ may depend on $F$.\\

\begin{assumption}\label{Assumptionp}\ \\
\ \\
\begin{enumerate}[label=A\arabic*']
\item\label{A: D1p} The function $\mathcal{L}\colon \ \RR^{p}\to\RR$ is convex. On $\mathrm{int}(\mathrm{dom}(F))$, it is differentiable and $\nabla \LLL$ is locally Lipschitz continuous. Moreover, $\mathcal{S} \neq \emptyset$.
\item\label{A: D3p} The initialization satisfies $\bb_{0}\in \mathrm{int}(\mathrm{dom}(F))$.
\item\label{A: D2p} The function $F\colon \R^p \to \RR\cup \left\{+\infty\right\}$ is Legendre. 
Moreover, $\dom (F)$ is closed, $F$ is twice-differentiable on $\mathrm{int}(\mathrm{dom}(F))$ and $\nabla^{2}F(\bb)$ is invertible for every $\bb\in \mathrm{int}(\mathrm{dom}(F))$. The map $\bb\in\mathrm{int}(\mathrm{dom}(F)) \mapsto (\nabla^{2}F(\bb))^{-1}$ is locally Lipschitz continuous. Moreover, if $z \in\dom (F)$ and $\beta_k \in \mathrm{int}\dom(F)$ for all $k\in\N$, we have that $D_F(z, \beta_k) \to 0$ if and only if $\beta_k \to z$. The function $G\colon\mathrm{int}(\mathrm{dom}(F))  \to \RR_{+}$ is locally Lipschitz continuous and there exists a constant $C>0$ such that $G(\bb)\geq C$ for every $\bb \in \mathrm{int}(\dom(F))$. 
 \item\label{A: D4p} There exists $\bar{z}\in \mathcal{S}$ such that the function $D_{F}(\bar{z},\cdot)$ is coercive.
\end{enumerate}
\end{assumption}\vspace{4mm}

\begin{theorem}\label{T: Existencep}
Under Assumption \ref{Assumptionp} the dynamical system \eqref{eq: x(t)flow} has a unique solution $\bb$ defined on $[0,+\infty)$.
Moreover, for every ${z}\in \mathcal{S}$, both the Bregman divergence $D_{F}(z,\bb(t))$ and $\mathcal{L}(\bb(t))$ are non-increasing functions of $t\geq 0$.
\end{theorem}
To prove Theorem \ref{T: Existencep}, we establish an auxiliary result. It states that both the Bregman distance and the loss function, along the trajectory of the flow, do not increase. 
\begin{lemma}\label{L: Loss-Decreasingp} Suppose that there exists a solution $\bb(t)$ of \ref{eq: x(t)flow} for $t \in \left[0,T\right)$. Suppose that Assumptions \ref{A: D2p} and \ref{A: D1p} hold, $\bb(t) \in \mathrm{int}(\dom (F))$ for every $t \in \left[0,T\right)$. Then, both the Bregman divergence $D_{F}(z,\bb(t))$ for every $z\in \mathcal{S}$ and the functional value $\mathcal{L}(\bb(t))$ are non-increasing functions for $t\in [0,T)$.
\end{lemma}\vspace{3mm}
Let $z\in\mathcal{S}$. To prove that $t\mapsto D_{F}(z,\bb(t))$ is decreasing, it is sufficient to note that, for every $t\in [0,T)$,
\begin{align}
\frac{d}{dt} D_{F}(z,\bb(t))=&\hspace{1mm}-\scal{\nabla F(\bb(t))}{\dot{\bb}(t)} -\scal{\nabla^{2} F(\bb(t))\dot{\bb}(t)}{z-\bb(t)}+\scal{\nabla F(\bb(t))}{\dot{\bb}(t)}\nonumber\\
=&\hspace{1mm}-\scal{\nabla^{2} F(\bb(t))\dot{\bb}(t)}{z-\bb(t)}\nonumber\\
=&\hspace{1mm}G(\bb(t))\scal{\nabla\mathcal{L}(\bb(t))}{z-\bb(t)}\label{eq: derivative DF1}\\ 
\leq& \ G(\bb(t))(\mathcal{L}(z)-\mathcal{L}(\bb(t)))\nonumber\\ \leq&\hspace{1mm} 0,\nonumber
\end{align}
where we used that $\bb$ is a solution of \eqref{eq: x(t)flow}, the gradient inequality for the convex function $\mathcal{L}$ and the fact that, for every $t\in [0,T)$,
 $\mathcal{L}(\bb(t))\geq \mathcal{L}(z)$ and $G(\bb(t))\geq 0$.
On the other hand, for $t\in [0,T)$, $\nabla^2 F(\bb(t))$ is invertible (and the inverse is positive definite). Then,
\begin{align}
    \frac{d}{dt} \mathcal{L}(\bb(t))=&\hspace{1mm}\scal{\nabla\mathcal{L}(\bb(t))}{\dot{\bb}(t)}\nonumber  \\=&\hspace{1mm}- G(\bb(t))\scal{\nabla\mathcal{L}(\bb(t))}{ \left[\nabla^2 F(\bb(t))\right]^{-1}\nabla\mathcal{L}(\bb(t))}\nonumber\\ \leq &\hspace{1mm}0,\nonumber
\end{align}
which completes the proof.

\begin{proof}[of Theorem \ref{T: Existencep}]\\
\textbf{Existence.} 
Denote by $T_{M}$ the maximal time for the existence of a solution of \eqref{eq: x(t)flow} in $\mathrm{int}(\dom(F))$; namely,
\begin{align}
        T_{M}:=\sup\left\{T>0\mid \exists \text{ a } \, C^1 \text{ function } \bb\colon [0,T)\to \mathrm{int}(\dom(F)) \text{ solution of \eqref{eq: x(t)flow} }\right\}.\nonumber
\end{align} 
    The proof of the fact $T_M >0$ is analogous to the one in Theorem~\ref{T: Existence}. Now we show that $T_{M}=+\infty$. Let $K$ be the closure of $\bb([0,T_M))$. Lemma \ref{L: Loss-Decreasingp} and Assumption~\ref{A: D4p} imply that $\bb([0,T_M))$ is bounded; therefore, $K$ is compact. Recall that the trajectory $\bb$ is continuous and that $\bb([0,T_M))$ is a subset of $\mathrm{int}(\dom(F))$. If $K\subset \mathrm{int}(\dom(F))$, by applying \cite[Proposition 1.57]{barreira2012ordinary}, we conclude that $T_{M}=+\infty$. Otherwise, there exists a point $\bar{\bb}\in \overline{\dom(F)} \setminus \mathrm{int}(\dom(F))$ for which there exists a sequence $(t_{j})_{j=1}^{+\infty}\subset \left[0,T_M\right)$ with $t_{j}\rightarrow T_{M}$ and $\bb(t_{j})\rightarrow \bar{\bb}$ for $j\to +\infty$. Integrating in time equation \eqref{eq: x(t)flow}, we get that, for every $j\in\NN$,
    \begin{align}
      \|\nabla F(\bb(t_{j}))-\nabla F(\bb_{0})\|=\left\|\int^{t_{j}}_{0}G(\bb(t))
\nabla\mathcal{L}(\bb(t))dt\right\|&\leq \int^{t_{j}}_{0}G(\bb(t))\left\|
\nabla\mathcal{L}(\bb(t))\right\|dt\nonumber\\ & \leq T_{M}\max_{\bb\in K}G(\bb)\max_{\bb\in K}\|\nabla\mathcal{L}(\bb)\|.\label{Bounded derivative}
\end{align} 
The maxima are well-defined due to the fact that the functions $G$ and $\norm{\nabla \mathcal{L}}$ are locally Lipschitz and so continuous, implying that they are bounded on the compact set $K$. For $j\to + \infty$, we have that $\bb(t_j) \to \bar{\bb} \in \overline{\dom(F)} \setminus \mathrm{int}(\dom(F))$ and therefore, since $F$ is a Legendre function, $\nabla F(\bb(t_{j}))\to+\infty$. Thus, the left-hand side of \eqref{Bounded derivative} tends to infinity as well, and this implies that $T_{M}=+\infty$.
 \ \\
\textbf{Uniqueness.} Since the right-hand side of \eqref{eq: x(t)flowI} is locally Lipschitz
at every point, local uniqueness holds for every initial state, and therefore globally.\hfill 
\end{proof}

\begin{remark}
    Note that in the previous theorem some of the hypothesis in Assumption \ref{Assumptionp} have not been used. For instance, we do not use that $\dom(F)$ is closed neither the continuity assumption on the Bregman divergence.
\end{remark}

\begin{theorem}\label{T: Main resultp}
Assume that Assumption \ref{Assumptionp} holds, then $\mathcal{L}(\bb(t))$ is non-increasing and $\mathcal{L}(\bb(t))\to \mathcal{L}^*$ for $t\to+\infty$. In addition, $\left[\mathcal{L}(\bb(t))-\mathcal{L}^{*}\right]\in L^{1}([0,+\infty))$, and, for every ${z}\in \mathcal{S}$ and $t > 0$,
\begin{align*}
  \LLL(\bb(t))-\LLL^{*} \quad \leq \quad \frac{D_{F}(z,\bb_0)}{C t }.
\end{align*}
Moreover, $\lim_{t\rightarrow +\infty}\bb(t)=\bb_{\infty}$ where
\begin{align}\label{minbreg}
\left\{\bb_{\infty}\right\} \ = \ \argmin\limits_{z\in\mathcal{S}} \left\{ D_F(z,\beta_0)+\int_{0}^{+\infty}G(\bb(t))\scal{\nabla\mathcal{L}(\bb(t))}{z}dt\right\}.
\end{align}
\end{theorem}
\vspace{2mm}
\begin{proof}
Let $z\in \mathcal{S}$. By the Fundamental Theorem of Calculus and inequality \eqref{eq: derivative DF1}, we have that, for all $t\geq 0$,
\begin{align}\label{breg_div1}
   D_{F}(z,\bb(t))-D_{F}(z,\bb_0)=&\int_{0}^{t}\frac{d}{ds}D_{F}(z,\bb(s)) \ ds\\
   \leq& -\int_{0}^{t}G(\bb(s))(\LLL(\bb(s))-\LLL^{*})ds \nonumber \\
   \leq& -C\int_{0}^{t}(\LLL(\bb(s))-\LLL^{*})ds
   \leq0\nonumber,
   \end{align} 
where the last inequality holds due to Assumption \ref{A: D2p}.
Since $D_{F}(z,\bb(t))$ is non-negative, from the above inequality we derive that, for every $t\geq 0$,
   \begin{equation*}
   \int_{0}^{t}(\LLL(\bb(s))-\LLL^{*})ds\leq \frac{D_{F}(z,\bb_0)}{C}.
  \end{equation*}
In particular, $t\mapsto \left[\LLL(\bb(t))-\LLL^{*}\right]$ belongs to $L^1\left(\left[0,+\infty\right)\right)$. Moreover, since $t\mapsto \LLL(\bb(t))$ is non-increasing by Lemma \ref{L: Loss-Decreasingp}, we have that $\left(\LLL(\bb(t))-\LLL^{*}\right)\cdot t \leq \int_{0}^t (\LLL(\bb(s))-\LLL^{*})ds\leq D_{F}(z,\bb_0)/C$ for every $t\geq 0$. From the latter, we obtain the rate stated in the theorem.\\

Next, we prove the convergence of the trajectory $(\bb(t))_{t>0}$. Let $\bar{z}$ be the point in Assumption \ref{A: D4p}. By Lemma \ref{L: Loss-Decreasingp}, $D_{F}(\bar{z},\bb(t))$ is (non-negative and) non-increasing, and therefore there exists $\lim\limits_{t\rightarrow +\infty}D_{F}(\bar{z},\bb(t))\in\RR$. Since $D_{F}(\bar{z},\cdot)$ is coercive by Assumption~\ref{A: D4p}, we have that $\bb(t)$ is bounded. Then, there exist a sequence $\left(t_{k}\right)_{k\geq0}$ and a point $\bb_*$ such that, for $k\to +\infty$, $t_k\to +\infty$ and $\bb(t_k)\to\bb_*$. Since $\bb(t_k)\in\mathrm{int}(\dom(F))$, and $\dom(F)$ is closed by assumption  \ref{A: D2p}, $\bb_*\in\dom(F)$. Moreover, since $\LLL$ is continuous and $\LLL(\bb(t))-\LLL^{*} \leq D_{F}(z,\bb_0)/(Ct)$, we deduce that 
\begin{equation*}\LLL(\bb_*)=\lim\limits_{k\to + \infty} \LLL(\bb(t_k))= \lim\limits_{t\to +\infty} \LLL(\bb(t))=\LLL^*;\end{equation*} 
   i.e., that $\bb_*\in\mathcal{S}$. We are going to show that indeed $\lim\limits_{t\rightarrow \infty}\bb(t)=\bb_*$. \\
Suppose that there exists another sequence $\left(t'_{n}\right)_{n=1}^{\infty}$ such that $t'_n\to +\infty$ and $\lim\limits_{n\to +\infty}\bb(t'_n) = \bar{\bb}$. With the same reasoning as above, we know that $\bar{\bb}$ belongs to $\mathcal{S}$. The existence of $\lim\limits_{t\to +\infty}D_{F}(\bb_*,\bb(t))$ yields
\begin{align}
 &\lim\limits_{n\to +\infty}D_{F}(\bb_*,\bb(t'_{n}))=\lim\limits_{t\to +\infty}D_{F}(\bb_*,\bb(t))=\lim\limits_{k\to +\infty}D_{F}(\bb_*,\bb(t_k))=0,\nonumber
\end{align}
where the last equality is given by Assumption \ref{A: D2p}. Again by Assumption \ref{A: D2p}, $\lim\limits_{n\to +\infty}D_{F}(\bb_*,\bb(t'_{n}))=0$ gives $\bar{\bb}=\bb_*$ and therefore $\lim\limits_{t\rightarrow +\infty}\bb(t)=\bb_*$.

\ \\
To conclude, note that, defining $\phi(z):=\lim\limits_{t\to +\infty} D_{F}(z,\bb(t))$, $\phi(z) \geq 0$ and $\phi(\beta_*)=0$. Then,
\begin{align}    \bb_*\in&\argmin\limits_{z\in\mathcal{S}} \ \phi(z). \nonumber
\end{align}
By \eqref{breg_div1} and \eqref{eq: derivative DF1}, 
\begin{align}
    \phi(z)\nonumber&=\lim\limits_{t\to +\infty} D_{F}(z,\bb(t))\nonumber= \lim\limits_{t\to +\infty} D_F(z,\bb_0)+\int_{0}^{t}G(\bb(s))\scal{\nabla\mathcal{L}(\bb(s))}{z-\bb(s)}ds\nonumber \\
    & = F(z)-F(\bb_0)-\scal{\nabla F(\bb_{0})}{z-\bb_0}+\int_{0}^{+\infty}G(\bb(s))\scal{\nabla\mathcal{L}(\bb(s))}{z-\bb(s)}ds.\nonumber
\end{align}
Discarding the terms that are constant with respect to $z$, we obtain the claim.\hfill
\end{proof}
\ \\
\ \\
An analogous result to Corollary \ref{Remark: Saving energy} is given by the following corollary. 
\begin{corollary}
Suppose that Assumptions \ref{Assumptionp} hold and let  $\bb$ be the unique solution of the dynamical system \eqref{eq: x(t)flow}. Let $X\in \RR^{d\times p}$ and let $\ell\colon\RR^d\to \RR$ be differentiable. Assume that $\LLL=\ell \circ X$ and $\mathcal{S} \subseteq \left\{\bb\in\RR^{p}\mid X\bb=y\right\}$ for some $y\in\RR^{d}$. Let $P: \ \RR^p \to \RR^p$ be the projection operator onto the kernel of $X$. Then, for every $t>0$,
\begin{align}
\frac{d}{dt}\left(P(\nabla F\left(\bb(t)\right)\right)&=0.
\nonumber
\end{align}
Moreover, $\lim\limits_{t\rightarrow+\infty}\bb(t)=\bb_{\infty}$ with
\begin{align*} \{ \bb_{\infty} \} \ = \ \argmin\limits_{z\in\mathcal{S}} \ D_F(z,\bb_0). 
\end{align*}
\end{corollary}
The point $\bb_{\infty}$ is the (unique) Bregman projection, with respect to the entropy $F$, of $\bb_0$ on $\mathcal{S}$ (see \cite{bauschke1997legendre}). In particular, $\bb_{\infty}$ belongs to the domain of the entropy $F$, which therefore can also implicitly enforce some constraints. For instance, 
the parameterization $\bb=\theta\odot\theta$ implicitly enforces the positivity of the solution: in this case the domain of $F$ is the positive orthant, see the case $L=2$ in equation~\eqref{eq:had}. 

\end{document}